\documentclass[10pt,english,a4paper]{amsart}
\pdfoutput=1

\usepackage[latin1]{inputenc}
\usepackage[T1]{fontenc}
\usepackage{lmodern}
\usepackage{babel}
\usepackage{amsmath,amssymb,amsthm}
\usepackage[marginratio={1:1,1:1},totalwidth=400pt,totalheight=600pt]{geometry}
\usepackage{microtype}
\usepackage{hyperref}
\usepackage{url}
\usepackage{color}

\newtheorem{theorem}{Theorem}[section]
\newtheorem{lemma}[theorem]{Lemma}
\newtheorem{proposition}[theorem]{Proposition}
\newtheorem{corollary}[theorem]{Corollary}

\theoremstyle{definition}
\newtheorem{definition}[theorem]{Definition}
\newtheorem{example}[theorem]{Example}
\newtheorem{question}[theorem]{Question}
\newtheorem{remark}[theorem]{Remark}

\theoremstyle{remark}

\numberwithin{equation}{section}

\newcommand{\Z}{\mathbb{Z}}
\newcommand{\N}{\mathbb{N}}
\newcommand{\T}{\mathbb{T}}
\newcommand{\F}{\mathbb{F}}
\newcommand{\C}{\mathbb{C}}
\newcommand{\R}{\mathbb{R}}
\DeclareMathOperator*{\aut}{Aut}

\newcommand{\op}{\operatorname}

\renewcommand\){\textup{)}}

\newcommand\K{\mathcal{K}}

\begin{document}

\title[Simplicity and uniqueness of trace for $C^*_r(G,\sigma)$]{On reduced twisted group C$^*$-algebras that are simple and/or have a unique trace}

\author[B\'edos]{Erik B\'edos}
\address{Department of Mathematics\\University of Oslo\\P.O.Box 1053 Blindern\\NO-0316 Oslo\\Norway}
\email{bedos@math.uio.no}
\author[Omland]{Tron Omland}
\address{Department of Mathematics\\University of Oslo\\P.O.Box 1053 Blindern\\NO-0316 Oslo\\Norway}
\email{trono@math.uio.no}



\date{June 5, 2017}

\begin{abstract}
We study the problem of determining when the reduced twisted group C$^*$-algebra associated with a discrete group $G$ is simple and/or has a unique tracial state, and present new sufficient conditions for this to hold. One of our main tools is a combinatorial property, that we call the relative Kleppner condition, which ensures that a quotient group $G/H$ acts by freely acting automorphisms on the twisted group von~Neumann algebra associated to a normal subgroup $H$. We apply our results to different types of groups, e.g.\ wreath products and Baumslag-Solitar groups.
\end{abstract}

\maketitle

\section{Introduction}
The theory of twisted group C$^*$-algebras is closely related to projective unitary representations of groups, and we refer to \cite{Pa2008} for a survey describing its importance in various fields of mathematics and physics. In this article, we will only consider discrete groups.
Simplicity and/or uniqueness of the trace for reduced twisted group C$^*$-algebras have been investigated in several papers, e.g.\ \cite{Sla, Pa, PR, Bed, Bed2, Bed4, BC2, BO, Om2}, and our aim with the present work is to provide better insight on this topic.
Finding new examples of simple C$^*$-algebras is always a valuable task, due to the role they play as building blocks and test objects. From the point of view of representation theory, simplicity of the reduced twisted group C$^*$-algebra $C_r^*(G,\sigma)$ gives interesting information as it amounts to the fact that any $\sigma$-projective unitary representation of $G$ which is weakly contained in the (left) regular $\sigma$-projective representation $\lambda_\sigma$ of $G$ is weakly equivalent to $\lambda_\sigma$. The reasoning behind this is essentially the same as the one given in \cite{H} in the untwisted case, i.e., when $\sigma$ is trivial. On the other hand, knowing that $C_r^*(G,\sigma)$ has a unique tracial state $\tau$ is also very useful. This property is a C$^*$-algebraic invariant in itself, which may be refined by taking into account the range of the restriction of $\tau$ to all projections in $C_r^*(G,\sigma)$. When $G$ is countable, this range is a countable subset of the interval $[0,1]$ (see \cite{Rie}), giving a way to label the gaps of the spectrum of self-adjoint elements in $C_r^*(G,\sigma)$.

We will let $G$ denote a group and $\sigma\colon G\times G\to \T$ a normalized $2$-cocycle on $G$ with values in the circle group $\T$, that is, $\sigma \in Z^2(G, \T)$.
We will often use the terminolgy introduced in \cite{BO} and say that the pair $(G,\sigma)$ is \emph{$C^*$-simple} (resp.\ has \emph{the unique trace property}) when the reduced twisted group C$^*$-algebra $C_r^*(G, \sigma)$ is simple (resp.\ has a unique tracial state).
If this holds when $\sigma $ is trivial, we will just say that $G$ is $C^*$-simple (resp.\ has the unique trace property), as in for example \cite{Bed, Bed2, Bed4, H, dHP, TD, OO, KK, BKKO, LB, Haa, Ken}.
We recall that if $(G, \sigma)$ is $C^*$-simple (resp.\ has the unique trace property), then $(G, \sigma)$ satisfies \emph{Kleppner's condition}, that is, every nontrivial $\sigma$-regular conjugacy class in $G$ is infinite (cf.~\cite{Kle} and subsection~\ref{Klep}).
In other words, setting
\begin{align*}
C^*S(G)&= \{\sigma \in Z^2(G,\T)\mid (G,\sigma) \ \text{is $C^*$-simple}\},\\
UT(G)&= \{\sigma \in Z^2(G,\T)\mid (G,\sigma) \ \text{has the unique trace property}\},\\
K(G)&= \{\sigma \in Z^2(G,\T)\mid \sigma \ \text{satisfies Kleppner's condition}\},
\end{align*}
we always have $C^*S(G) \subset K(G)$ and $UT(G) \subset K(G)$.
Following \cite{BO}, we will let $\K_{C^*S}$ (resp.~$\K_{UT}$) denote the class of groups $G$ satisfying $C^*S(G) = K(G)$ (resp.~$UT(G) = K(G)$).
Moreover, $\K$ will denote the intersection of $\K_{C^*S} $ and $\K_{UT}$. Thus, if $G$ belongs to $\K$, then for any $\sigma \in Z^2(G,\T)$, we have that $(G,\sigma)$ is $C^*$-simple if and only if $(G,\sigma)$ has the unique trace property, if and only if $(G,\sigma)$ satisfies Kleppner's condition.

It is noteworthy that the class $\K$ contains many amenable groups. Finite groups, abelian groups, FC-groups and nilpotent groups all lie in $\K$, and more generally, as shown in \cite{BO}, every \emph{FC-hypercentral} group belongs to $\K$ (cf.~subsection~\ref{FC-hyp}). On the other hand, it is known \cite{PR} that some semidirect products of $\Z^n$ by actions of $\Z$ do not belong to $\K_{C^*S}$ (and neither to $\K_{UT}$).
In a somewhat opposite direction, Bryder and Kennedy have recently shown \cite{BK} that $C^*S(G) = Z^2(G,\T)$ (resp.~$UT(G)= Z^2(G,\T)$) whenever $G$ is $C^*$-simple (resp.\ has the unique trace property). Since the class of $C^*$-simple groups is (strictly) contained in the class of groups with the unique trace property \cite{BKKO, LB}, we get that every $C^*$-simple group belongs to $\K$, while every group with the unique trace property belongs to $\K_{UT}$.
Combining results from \cite{BO} and \cite{BK}, we show in the present paper that a group $G$ belongs to $\K_{UT}$ whenever the \emph{FC-hypercenter of $G$} coincides with its amenable radical (cf.~Theorem~\ref{FCH-P}). An interesting question is whether this property in fact characterizes $\K_{UT}$.

When some $\sigma \in Z^2(G,\T)$ is given and it is unclear whether $G$ lies in $\K_{C^*S}$, or in $\K_{UT}$, one would like to be able to decide whether $\sigma$ lies in $C^*S(G)$, or in $UT(G)$. Our main contribution is to provide several new conditions that are sufficient to handle many cases.
As the first step in our approach, we consider a normal subgroup $H$ of $G$ and study when certain naturally arising $*$-automorphisms of the twisted group von~Neumann algebra $M$ associated to $H$ are \emph{freely acting} (or~\emph{properly outer}) in the sense of \cite{Kal}.
This leads us to introduce a combinatorial property for a triple $(G,H, \sigma)$, that we call \emph{the relative Kleppner condition}, which ensures that the canonical twisted action of the quotient group $G/H$ on the von~Neumann algebra $M$ is freely acting. Combining this property with 
some results from \cite{Bed, Bed4} and building on previous works of Kishimoto in \cite{Kis} and Olesen and Pedersen in \cite{OP}, we obtain some conditions that are sufficient for $\sigma$ to belong to $UT(G)$, or to $C^*S(G) \cap UT(G)$.
We illustrate the usefulness of these conditions by applying them to a variety of groups (e.g.\ semidirect products, wreath products, and Baumslag-Solitar groups).

The paper is organized as follows. Section~\ref{prelim} contains a review of the definitions and of the results that are relevant for this article.
In Section~\ref{looking-subgroups} we look at the behavior of $C^*$-simplicity and the unique trace property for pairs $(G,\sigma)$ in a few basic group constructions, in particular in connection with subgroups.
Section~\ref{free} is devoted to freely acting automorphisms and the relative Kleppner condition for triples $(G,H, \sigma)$. 
Our main result is Theorem~\ref{utp}, which relies on some technically involved arguments, in particular in the proof of Proposition~\ref{freeact}. Theorem~\ref{utp} has several consequences; especially, it implies that $C^*$-simplicity and the unique trace property pass from $(H,\sigma_{|H\times H})$ to $(G,\sigma)$ whenever $(G,H, \sigma)$ satisfies the relative Kleppner condition.
Section~\ref{ex} contains a detailed study of several new examples. First we discuss semidirect products of abelian groups by aperiodic automorphisms. Next we look at wreath products, with special focus on $\Z \wr \Z$ and $\Z_2 \wr \Z$, where the former requires investigation of the noncommutative infinite-dimensional torus, and the latter gives rise to a noncommutative version of the lamplighter group. Then we discuss a semidirect product arising from the Sanov action of $\F_2$ on $\Z^2$. Finally, we consider the Baumslag-Solitar groups.

We will often refer to the fact that if $G$ is amenable, or if $G$ is exact and $C_r^*(G,\sigma)$ has stable rank one, then $(G,\sigma)$ is $C^*$-simple whenever it has the unique trace property (cf.~Theorem~\ref{Murph}). For completeness, adapting some previous work of Dykema and de~la~Harpe \cite{DH} for reduced group C$^*$-algebras, we discuss in Appendix~\ref{appendix} some conditions ensuring that $C_r^*(G,\sigma)$ has stable rank one.
In Appendix~\ref{BP} we prove a twisted version of Tucker-Drob's unpublished result in \cite{TD} saying that a group has the unique trace property whenever it has the so-called property~\(BP\). Finally, in Appendix~\ref{decay}, we generalize Gong's recent result in \cite{Gong1} by showing how decay properties of $(G,\sigma)$ can be combined with superpolynomial growth of nontrivial $\sigma$-regular classes to deduce uniqueness of the trace. 

\section{Preliminaries and known results}\label{prelim}

\subsection{\texorpdfstring{$2$}{2}-cocycles}\label{cocy}
Throughout this paper, $G$ will denote a (discrete) group with identity $e$, while $\sigma$ will denote a normalized $2$-cocycle (sometimes called a multiplier) on $G$ with values in the circle group $\T$, as in \cite{ZM}.
This means that we have $\sigma(g, e) = \sigma(e, g) = 1$ for every $g \in G$ and that the cocycle identity
\begin{equation}\label{coceq}
\sigma(g, h)\sigma(gh, k) = \sigma(h, k) \sigma(g, hk)
\end{equation}
holds for every $g, h, k \in G$.

The set $Z^2(G,\T)$ of all normalized $2$-cocycles becomes an abelian group under pointwise product, the inverse operation corresponding to conjugation, i.e., $\sigma ^{-1} = \overline{\sigma}$, where $\overline{\sigma}(g, h) = \overline{\sigma(g, h)}$, and the identity element being the trivial $2$-cocycle $1$ on $G$.

An element $\beta \in Z^2(G,\T)$ is called a \emph{coboundary} whenever we have
\[
\beta(g,h)= b(g)b(h)\overline{b(gh)}
\]
for all $g,h\in G$, for some $b\colon G \to \T$ such that $b(e)=1$ (such a function $b$ is uniquely determined up to multiplication by a character of $G$).
The set of all coboundaries $B^2(G,\T)$ is a subgroup of $Z^2(G,\T)$,
and elements in the quotient group $H^2(G,\T)=Z^2(G,\T)/B^2(G,\T)$ will be denoted by $[\sigma]$.

For $\sigma, \omega \in Z^2(G,\T)$, we write $\sigma \sim \omega$ and say that $\sigma$ is \emph{similar} (or \emph{cohomologous}) to $\omega$ when $[\sigma] = [\omega]$ in $H^2(G,\T)$.

\subsection{Twisted group algebras}
The \emph{left regular $\sigma$-projective unitary representation $\lambda_{\sigma}$ of $G$ on $\mathcal{B}(\ell^2(G))$} is given by
\[
\big(\lambda_{\sigma}(g)\xi\big)(h)=\sigma(g,g^{-1}h)\, \xi(g^{-1}h)
\]
for $g,h \in G$ and $\xi \in \ell^2(G)$.
Note that we have
\begin{align*}
\lambda_{\sigma}(g)\, \delta_h&=\sigma(g,h)\, \delta_{gh}\,,\\
\lambda_{\sigma}(g)\, \lambda_{\sigma}(h)&=\sigma(g,h)\, \lambda_{\sigma}(gh)
\end{align*}
for all $g,h\in G$, where $\delta_h(g)=1$ if $g=h$ and $\delta_h(g)=0$ otherwise.
It follows that for all $g, h \in G$ we have
\begin{equation*}
\lambda_{\sigma}(g)\lambda_{\sigma}(h)\lambda_{\sigma}(g)^*= \sigma(g,h)\,\overline{\sigma(ghg^{-1},g)} \,\lambda_{\sigma}(ghg^{-1})\,.
\end{equation*}
We will use the notation $g\cdot h := ghg^{-1}$ to denote the action of $G$ on itself by conjugation.
Letting $\widetilde{\sigma}\colon G\times G \to \C $ denote the anti-symmetrized form of $\sigma$ defined by 
\begin{equation}\label{t-sigma}
\widetilde{\sigma}(g,h) =\sigma(g,h)\,\overline{\sigma(g\cdot h,g)}\,,
\end{equation}
we get
\begin{equation}\label{sigma-conjugate}
\lambda_{\sigma}(g)\lambda_{\sigma}(h)\lambda_{\sigma}(g)^*= \widetilde\sigma(g,h) \,\lambda_{\sigma}(g\cdot h)
\end{equation}
for all $g, h \in G$.

The \emph{reduced twisted group C$^*$-algebra} $C^*_r(G,\sigma)$ and the \emph{twisted group von~Neumann algebra} $W^*(G,\sigma)$ are, respectively, the C$^*$-algebra and the von~Neumann algebra generated by $\lambda_{\sigma}(G)$. We will use the convention that when $\sigma$ is the trivial cocycle, we just drop $\sigma$ from all our notation. It is well-known and easy to check that $C^*_r(G,\sigma)\simeq C^*_r(G,\omega)$ (resp.~$W^*(G,\sigma)\simeq W^*(G,\omega))$ whenever $\sigma \sim \omega$ in $Z^2(G, \T)$.

We will denote by $\Lambda_\sigma$ the map from $\ell^1(G)$ into $\mathcal{B}(\ell^2(G))$ given by
\[
\Lambda_\sigma(f) = \sum_{g\in G} \,f(g)\, \lambda_\sigma(g)
\]
for $f \in \ell^1(G)$.
Note that for $f \in \ell^1(G)$ and $\xi \in \ell^2(G)$, we have $\Lambda_\sigma(f) \, \xi = f \ast_\sigma \xi$, where
\[
(f \ast_\sigma \xi)(h) = \sum_{g\in G}\, f(g) \, \xi(g^{-1}h) \sigma(g, g^{-1}h)
\]
for each $h\in G$.

The canonical tracial state on $W^*(G,\sigma)$ will be denoted by $\tau$ (or by $\tau_\sigma$ if confusion may arise); it is given as the restriction to $W^*(G,\sigma)$ of the vector state associated with $\delta_e$. As is well-known, $\tau$ is faithful and satisfies $\tau(\lambda_\sigma(g)) = 0$ for every $g\neq e$. The restriction of $\tau$ to $C^*_r(G,\sigma)$ will also be denoted by $\tau$ (or by $\tau_\sigma$).

Note that one can also consider the right regular $\sigma$-projective unitary representation $\rho_{\sigma}$ of $G$ on $\mathcal{B}(\ell^2(G))$ given by
\[
(\rho_{\sigma}(g)\xi)(h)=\sigma(h,g)\,\xi(hg)
\]
for all $\xi\in\ell^2(G)$ and $g,h \in G$.
One easily checks (see e.g.\ \cite[Section~2]{Om}) that for every $g, h \in G$ we have
\[
\lambda_\sigma(g) \, \rho_{\overline{\sigma}}(h) = \rho_{\overline{\sigma}}(h)\, \lambda_\sigma(g)\,.
\]
We will say that $(G,\sigma)$ is \emph{$C^*$-simple} (resp.\ has \emph{the unique trace property}) whenever $C_r^*(G,\sigma)$ is simple (resp.\ $\tau$ is the only tracial state of $C_r^*(G,\sigma)$).

\subsection{Kleppner's condition}\label{Klep}

We recall \cite{Kle, Pa, Om} that $g \in G$ is called \emph{$\sigma$-regular} if
\[
\sigma(h,g)=\sigma(g,h) \, \, \text{whenever} \, \, h \in G\, \, \text{commutes with}\, \, g\,.
\]
If $g$ is $\sigma$-regular, then $kgk^{-1}$ is $\sigma$-regular for all $k$ in $G$,
so the notion of $\sigma$-regularity makes sense for conjugacy classes in $G$.

Following \cite{Om}, we will say that the pair $(G,\sigma)$ satisfies \emph{Kleppner's condition} (or~\emph{condition~K}) if every nontrivial $\sigma$-regular conjugacy class of $G$ is infinite. It is known \cite{Kle, Pa, Om} that $(G,\sigma)$ satisfies Kleppner's condition if and only if $W^*(G,\sigma)$ is a factor, if and only if $C^*_r(G,\sigma)$ has trivial center, if and only if $C^*_r(G,\sigma)$ is prime.

It follows easily from the above equivalences that Kleppner's condition is necessary for $(G,\sigma)$ to be $C^*$-simple (resp.\ to have a unique trace). However, in general, Kleppner's condition is not sufficient for any of these two properties to hold. For instance, if $G$ is a nontrivial amenable group which is ICC (i.e., every nontrivial conjugacy class in $G$ is infinite), then $(G,1)$ satisfies Kleppner's condition, but $(G,1)$ is neither $C^*$-simple, nor has a unique tracial state (since there exists a nontrivial homomorphism $\epsilon\colon C_r^*(G)\to \C$ whenever $G$ is amenable, cf.~\cite{BrOz}).

Recall from the introduction that 
\begin{align*}
C^*S(G)&= \{\sigma \in Z^2(G,\T)\mid (G,\sigma) \ \text{is $C^*$-simple}\}\,,\\
UT(G)&= \{\sigma \in Z^2(G,\T)\mid (G,\sigma) \ \text{has the unique trace property}\}\,,\\
K(G)&= \{\sigma \in Z^2(G,\T)\mid \sigma \ \text{satisfies Kleppner's condition}\}\,.
\end{align*}
We then have
\[
C^*S(G) \subset K(G)\quad \text{and} \quad UT(G) \subset K(G)\,.
\]
It is straightforward to see that if $\sigma$ lies in $C^*S(G)$ (resp.~$UT(G)$) and $\omega\in Z^2(G,\T)$ is similar to $\sigma$, then $\omega$ also lies in $C^*S(G)$ (resp.~$UT(G)$). Hence, it follows from \cite{BKKO} that if $\sigma \sim 1$ and $\sigma \in C^*S(G)$, then $\sigma \in UT(G)$. We do not know whether this implication holds when $\sigma \not\sim 1$.
Note that it may happen that $K(G)$ is empty, in which case $C^*S(G)$ and $UT(G)$ are also empty.
For example, suppose that $G$ is finite and that there exists some $\sigma \in K(G)$.
Then $W^*(G,\sigma)$ is a finite-dimensional factor having a basis indexed by $G$.
So $\lvert G \rvert$, the cardinality of $G$, has to be a square number.
Thus, $K(G) = \emptyset $ whenever $G$ is finite and $\lvert G \rvert$ is not a square number.
We also note that $K(\Z)=\emptyset$, as readily follows from the fact that $H^2(\Z, \T)$ is trivial.
Another fact which is almost immediate is that $G$ is ICC if and only if $K(G) = Z^2(G,\T)$.

We will say that $G$ belongs to the class $\K_{C^*S}$ if $C^*S(G) = K(G)$, and that $G$ belongs to the class $\K_{UT}$ if $UT(G) = K(G)$.
Moreover, $\K$ will denote the intersection of $\K_{C^*S} $ and $\K_{UT}$.

Finally, we mention that it follows from \cite{HZ} that $\sigma\in C^*S(G)\cap UT(G)$ if and only if $C^*_r(G,\sigma)$ has the Dixmier property relative to $\C\cdot 1$, if and only if $\sigma\in K(G)$ and $C^*_r(G,\sigma)$ has the Dixmier property relative to its center (as defined for example in \cite[III.2.5.16]{Bla}).

\subsection{Murphy's theorem}
A useful consequence of a result due to Murphy \cite{Mur} is the following theorem (cf.\ Corollaries~2.3 and~2.4 in \cite{BO}):
\begin{theorem}\label{Murph}
Assume that $G$ is amenable, or that $G$ is exact and $A=C^*_r(G, \sigma)$ has stable rank one \(i.e., the invertible elements of $A$ are dense in $A$\). Then $(G,\sigma)$ is $C^*$-simple whenever it has the unique trace property.
\end{theorem}
This result implies that if $G$ is amenable, then $UT(G) \subset C^*S(G)$.
Hence, an amenable group belongs to $\K$ if and only it belongs to $\K_{UT}$. When $G$ is a countable and amenable, and $(G,\sigma)$ has the unique trace property, one can conclude from Theorem~\ref{Murph} that $C_r^*(G, \sigma)$ is a separable, simple, nuclear C$^*$-algebra with a unique tracial state, hence belongs to a class of C$^*$-algebras being currently under intensive study.

Concerning exactness of groups, the reader may consult \cite{BrOz} and references therein. When $\sigma \not \sim 1$, there are few known examples of pairs $(G, \sigma)$ such that $C^*_r(G, \sigma)$ has stable rank one. Putnam's result \cite{Put} for irrational rotation algebras deals with the case where $G= \Z^2$ (after rewriting rotation algebras as a twisted group C$^*$-algebras associated to $\Z^2$). His result was generalized to $G=\Z^n$ for any $n\geq 2$ by Blackadar, Kumjian and R{\o}rdam \cite{BKR}, but one should note that they effectively use simplicity to deduce stable rank one.

\begin{question}\label{simp-sr1}
Suppose $G$ is exact, $\sigma \in Z^2(G,\T)$ and consider the following statements:
\begin{itemize}\itemsep2pt
\item[(i)] $(G,\sigma)$ is $C^*$-simple.
\item[(ii)] $C_r^*(G,\sigma)$ has stable rank one and $(G,\sigma)$ has the unique trace property.
\end{itemize}
Theorem~\ref{Murph} gives that (ii)~$\Rightarrow$~(i). Does (i)~$\Rightarrow$~(ii) always hold ?

\medskip

If $\sigma \sim 1$, thanks to \cite{BKKO}, this question reduces to asking whether $C_r^*(G)$ has stable rank one whenever $G$ is $C^*$-simple (and $G$ is exact). More generally, one may wonder if $C_r^*(G,\sigma)$ has stable rank one whenever $(G,\sigma)$ is $C^*$-simple.
\end{question}
Adapting the approach used in \cite{DH}, where several groups whose reduced group C$^*$-algebras have stable rank one are presented, we discuss in Appendix~\ref{appendix} of this paper some conditions ensuring that $C_r^*(G,\sigma)$ has stable rank one.
\subsection{FC-hypercentral groups}\label{FC-hyp} 
 It is known that a group $G$ has a smallest normal subgroup that produces an ICC quotient group (cf.~\cite[Remark~4.1]{Jaw} and \cite[Proposition~2.5]{BO}). This subgroup coincides with the so-called \emph{FC-hypercenter} \cite{Robinson} of $G$ and is denoted by $FCH(G)$. Clearly, $FCH(G)=\{e\}$ if and only if $G$ is ICC. Letting $Z(G)$ denote the center of $G$ and $FC(G)$ the FC-center of $G$ (that is, the (normal) subgroup of $G$ consisting of all elements of $G$ having a finite conjugacy class in $G$), we have
\[
Z(G) \subset FC(G) \subset FCH(G).
\]
When $G= FCH(G)$, $G$ is said to be \emph{FC-hypercentral}.
Every FC-hypercentral group is amenable \cite{Ech}.
It follows that the FC-hypercenter of a group $G$ is amenable, so we have
\[
FCH(G) \subset AR(G),
\]
where $AR(G)$ denotes the \emph{amenable radical} of $G$, that is, the largest normal amenable subgroup of $G$.
Alternatively, one may deduce this inclusion by observing that $G/AR(G)$ has no normal amenable subgroup other than the trivial one, hence is ICC.

\begin{theorem}[\cite{BO}]\label{FCH}
Assume that $G$ is FC-hypercentral. Then $G$ belongs to $\K$.
\end{theorem}
We do not know of any amenable group that belongs to $\K$ without being FC-hypercentral.

\subsection{\texorpdfstring{$C^*$}{C*}-simple groups and groups with the unique trace property}
We refer to \cite{H} for a thorough introduction to $C^*$-simple groups and groups with the unique trace property. Among the many recent articles 
dealing with such groups, we mention \cite{dHP, TD, KK, BC2, OO, Gong1, Gong2, BKKO, LB, Ken, Haa, BK, LBB, Iva-Om}. As already pointed out in the introduction, it is now known from \cite{BKKO, LB} that the class of $C^*$-simple groups is strictly contained in the class of groups with the unique trace property. Another interesting result from \cite{BKKO} is that a group has the unique trace property if and only if its amenable radical is trivial. Moreover, if $G$ is $C^*$-simple (resp.~has the unique trace property), then $(G,\sigma)$ is $C^*$-simple (resp.~has the unique trace property) for every $\sigma \in Z^2(G,\T)$, as shown in \cite{BK}. It follows that the class of $C^*$-simple groups is (strictly) contained in $\K$ and that the class of groups with the unique trace property is (strictly) contained in $\K_{UT}$.

A very large family of groups with the unique trace property is the class of groups having the property~\(BP\) introduced
in \cite{TD}. As the proof of this fact, which relies on some arguments from \cite{BCH}, is only very briefly sketched in \cite[Remark~5.9]{TD}, we prove in Appendix \ref{BP} that $(G,\sigma)$ has the unique trace property whenever $G$ has property~\(BP\).

In \cite{Iva-Om}, the authors consider (nondegenerate) free products of groups with amalgamation. They give (in \cite[Section~4]{Iva-Om}) an example of such a group $\Gamma=G_0\ast_HG_1$ which is not $C^*$-simple, but is a so-called weak$^*$~Powers group, hence has property~\(BP\) (cf.~\cite[Theorem~5.4]{TD}). In particular, $\Gamma$ has the unique trace property. Moreover, as $G_0$ and $G_1$ are easily seen to be amenable, hence exact, $\Gamma$ is also exact (cf.~\cite{Dyk}). It therefore follows from Theorem~\ref{Murph} that $C_r^*(\Gamma)$ does not have stable rank one.

In another direction, Gong has recently shown in \cite[Theorem~3.11]{Gong1} that if a group $G$ has property RD \cite{Jolissaint} with respect to some length function $L$, and every nontrivial conjugacy class of $G$ has superpolynomial growth (w.r.t.~$L$), then $G$ has the unique trace property. This result applies for example when $G$ is a torsion-free, non-elementary, Gromov hyperbolic group, see \cite{Gong1, Gong2}. Such groups are in fact well-known to be $C^*$-simple, cf.~\cite{H}. We show in Appendix \ref{decay} how Gong's result may be generalized by considering suitable decay properties for a pair $(G,\sigma)$ in combination with superpolynomial growth of $\sigma$-regular classes.

\section{Looking at subgroups}\label{looking-subgroups}

\subsection{Subgroups and normal subgroups}\label{subgroups}

Let $H$ be a subgroup of $G$ and let $\sigma'$ denote the restriction of $\sigma$ to $H\times H$. We will denote the canonical tracial state on $W^*(H, \sigma')$ (resp.~$C_r^*(H, \sigma')$) by $\tau'$. It follows from \cite[subsection~4.26]{ZM} that there is a natural embedding of $W^*(H, \sigma')$ (resp.~$C_r^*(H, \sigma')$) into $W^*(G, \sigma)$ (resp.~$C^*_r(G, \sigma)$), sending $\lambda_{\sigma'}(h)$ to $\lambda_\sigma(h)$ for each $h \in H$.

We will usually identify $W^*(H, \sigma')$ (resp.~$C_r^*(H, \sigma')$) with its canonical copy inside $W^*(G, \sigma)$ (resp.~$C^*_r(G, \sigma)$). We note that there exists a normal conditional expectation $\mathcal{E}$ from $W^*(G, \sigma)$ onto $W^*(H, \sigma')$, satisfying $\mathcal{E}(\lambda_\sigma(g)) = \lambda_\sigma(g) $ if $g\in H$, and $\mathcal{E}(\lambda_\sigma(g)) = 0$ otherwise. Indeed, since the characteristic function of $H$ in $G$ is positive definite, the existence of a normal completely positive map $\mathcal{E}$ with this property follows for example from \cite{BC} (see Proposition~4.2 and Corollary~4.4 therein). It is then straightforward to check that this map is a conditional expectation. We will also use that the restriction of $\mathcal{E}$ to $C^*_r(G, \sigma)$ gives a conditional expectation $E$ from $C^*_r(G, \sigma)$ onto $C_r^*(H, \sigma')$.

When $H$ is a normal subgroup of $G$, the relationship between $W^*(G, \sigma)$ and $W^*(H, \sigma')$ (resp.\ between $C^*_r(G, \sigma)$ and $C^*_r(H, \sigma')$), may be described as follows, cf.~\cite{Bedos} (resp.~\cite{Bed}). 
First we note that equation \eqref{sigma-conjugate} implies that for each $g \in G$, the inner automorphism of $W^*(G, \sigma)$ (resp.~$C^*_r(G, \sigma)$) implemented by the unitary $\lambda_{\sigma}(g)$ restricts to a $*$-automorphism $\gamma_g$ of $W^*(H,\sigma')$ (resp.~$C^*_r(H, \sigma')$) satisfying
\[
\gamma_g\big(\lambda_{\sigma'}(h)\big)=\widetilde\sigma(g,h)\,\lambda_{\sigma'}(g\cdot h) \quad \text{ for each } h \in H\,.
\]
Let $q$ denote the canonical homomorphism from $G$ onto $K:=G/H$, let $s\colon K\to G$ be a section for $q$ satisfying $s(e)=e$,
and define $m\colon K\times K \to H$ by
\[
m(k,l)=s(k)s(l)s(kl)^{-1}\,.
\]
Moreover, define $\beta\colon K\to \aut(W^*(H, \sigma'))$ (resp.~$\aut(C^*_r(H, \sigma'))$) by
\[
\beta_k = \gamma_{s(k)} \quad \text{for each } \, k \in K,
\]
and $\omega\colon K \times K\to {\mathcal U}\big(C^*_r(H, \sigma')\big) \subset \,{\mathcal U}\big(W^*(H, \sigma')\big)$ by
\[
\omega(k,l) = \sigma\big(s(k),s(l)\big)\, \overline{\sigma\big(m(k,l),s(kl)\big)}\, \lambda_{\sigma'}(m(k,l))
\]
for each $k, l \in K$.
Then $(\beta,\omega)$ is a twisted action (sometimes called a cocycle crossed action) of $K$ on $W^*(H, \sigma')$ (resp.~$C^*_r(H, \sigma')$) such that 
\begin{align*}
W^*(G,\sigma) \, &\simeq \, W^*(H,\sigma')\rtimes_{(\beta,\omega)} K \\
\text{(resp.} \quad C^*_r(G,\sigma) \, &\simeq \, C^*_r\big(C^*_r(H,\sigma'),K,\beta,\omega\big))\,,
\end{align*}
cf.~\cite[Theorem~1]{Bedos} (resp.~\cite[Theorem~2.1]{Bed}). It should be noted that a similar decomposition result was first established for full twisted group C$^*$-algebras and full twisted crossed products by Packer and Raeburn in \cite[Theorem~4.1]{PR}.

When there is danger of confusion, we will denote each $\beta_k$ by $\beta^r_k$ when we consider it as a $*$-automorphism of $C^*_r(H, \sigma')$, and denote the associated twisted action of $K$ by $(\beta^r, \omega)$.
We note that the canonical tracial state $\tau'$ of $C_r^*(H, \sigma')$ is \emph{invariant under $\beta^r$}, that is, we have $\tau'\circ \beta_k^r = \tau'$ for each $k \in K$. This may be verified by direct computation on the generators of $C_r^*(H, \sigma')$. Alternatively, we may use that $\tau'$ is the restriction of $\tau$ to $C_r^*(H, \sigma')$ and observe that the restriction to $C_r^*(H, \sigma')$ of any tracial state of $C_r^*(G, \sigma)$ is invariant under $\beta^r$, since each $\beta^r_k$ is implemented by a unitary in $C_r^*(G, \sigma)$, namely $\lambda_\sigma(s(k))$.

For simplicity, we will just say that a tracial state of $C_r^*(H, \sigma')$ is \emph{$K$-invariant} when it is invariant under $\beta^r$. We will also say that \emph{$K$ acts on $C_r^*(H,\sigma')$ in a minimal way} when the zero ideal is the only proper (two-sided, closed) ideal of $C_r^*(H,\sigma')$ which is invariant under $\beta_k^r$ for each $k \in K$.

Using the decomposition $C_r^*(G,\sigma) \, \simeq \, C_r^*\big(C_r^*(H, \sigma'), K, \beta^r,\omega\big)$,
the following proposition is an immediate consequence of Bryder and Kennedy's recent results \cite[Corollaries~1.2 and~1.4]{BK}.
\begin{proposition}\label{BryKed}
Assume $H$ is normal and $K=G/H$.
\begin{itemize}
\item[(i)] If $K$ is $C^*$-simple, then $(G,\sigma)$ is $C^*$-simple if and only if $K$ acts on $C_r^*(H,\sigma')$ in a minimal way.
\item[(ii)] If $K$ has the unique trace property, then $(G,\sigma)$ has the unique trace property if and only if $\tau'$ is the only $K$-invariant tracial state of $C_r^*(H,\sigma')$.
\end{itemize}
\end{proposition}

\begin{remark}\label{Prim}
When $C_r^*(H,\sigma')$ is abelian, one may investigate if $K$ acts minimally by computing first the Gelfand spectrum of $C_r^*(H,\sigma')$, as we will do in Example~\ref{F2xZ-2} and Proposition~\ref{BS-prop}. More generally, one may try to determine $\op{Prim}(C_r^*(H,\sigma'))$, the primitive ideal space of $C_r^*(H,\sigma')$ equipped with the hull-kernel topology, and use the fact that there is a one-to-one correspondence between the ideals of a C$^*$-algebra $A$ and the closed subsets of $\op{Prim}(A)$ (see e.g.\ \cite[Section~II.6.5]{Bla}).
If $A$ is unital, then $\op{Prim}(A)$ is compact and the Dauns-Hofmann theorem provides an isomorphism between the center $Z(A)$ of $A$ and $C(\op{Prim}(A))$.
Thus, in the special case where $A$ is unital and $\op{Prim}(A)$ is Hausdorff, $\op{Prim}(A)$ is homeomorphic to the Gelfand spectrum of $Z(A)$.
We will illustrate how this may used in combination with Proposition~\ref{BryKed} in subsection~\ref{Z2xF2}.
\end{remark}

\subsection{Subgroups of finite index}
It is known \cite{BdH, Popa} that if $G$ is an ICC group and $H$ is a subgroup of $G$ with finite index, then we have
\begin{equation} \label{EI1}
\text{$G$ is $C^*$-simple } \Longleftrightarrow \text{ $H$ is $C^*$-simple}
\end{equation}
and
\begin{equation} \label{EI2}
\text{$G$ has the unique trace property } \Longleftrightarrow \text{ $H$ has the unique trace property}.
\end{equation}

Note that $H$ is ICC whenever $G$ is ICC and $[G:H] < \infty$.
In the twisted case, Kleppner's condition is not necessarily inherited by a subgroup of finite index.
A twisted version of \eqref{EI1} and \eqref{EI2} is therefore as follows.

\begin{proposition}\label{finite index}
Let $H$ be a subgroup of $G$ with finite index. Let $\sigma \in Z^2(G,\T)$ and let $\sigma'$ denote the restriction of $\sigma$ to $H\times H$. Assume that both $(G,\sigma)$ and $(H,\sigma')$ satisfy Kleppner's condition.
Then we have
\begin{equation} \label{TEI1}
\text{$(G,\sigma)$ is $C^*$-simple } \Longleftrightarrow \text{ $(H,\sigma')$ is $C^*$-simple}
\end{equation}
and
\begin{equation} \label{TEI2}
\text{$(G,\sigma)$ has the unique trace property }
\Longleftrightarrow
\text{ $(H,\sigma')$ has the unique trace property}.
\end{equation}
\end{proposition}

\begin{proof}
We will deduce both equivalences from \cite[Corollary~4.6]{Popa}, so we have to check that all the assumptions in this corollary are satisfied. We first recall that the GNS-representation of $A:= C_r^*(G,\sigma)$ corresponding to $\tau$ is the identity representation of $A$ on $\ell^2(G)$. The canonical conditional expectation $E$ from $A$ onto $B:=C_r^*(H,\sigma')$ (identified as a unital C$^*$-subalgebra of $A$) clearly satisfies that $\tau = \tau\circ E$. Since $(G,\sigma)$ satisfies Kleppner's condition (by assumption), we know that $W^*(G,\sigma)$ is a factor, hence that $\tau$ is factorial. Moreover, since $\tau_{\mid B}$ coincides with the canonical tracial state $\tau'$ of $B$, and $(H,\sigma')$ is assumed to satisfy Kleppner's condition, we also know that $\tau_{\mid B}$ is factorial. As explained in above, there exists a conditional expectation $\mathcal{E}$ from $W^*(G, \sigma)$ onto $W^*(H,\sigma')$ that extends $E$.

Now, let $\{ g_1, \cdots, g_n\}$ be a set of left coset representatives of $H$ in $G$. Then $\{\lambda_\sigma(g_i), \lambda_\sigma(g_i)^*\}_{i=1}^n$ is a quasi-basis for $E$ in the sense of \cite[Definition~1.2.2]{Wat}, that is, we have
\[
\sum_{i=1}^n \, \lambda_\sigma(g_i) \,E\big(\lambda_\sigma(g_i)^*\, x\big) = \, x \, = \sum_{i=1}^n\, E\big(x\, \lambda_\sigma(g_i)\big) \, \lambda_\sigma(g_i)^*
\]
for all $x \in A$. Indeed, by a density argument, it suffices to show that this holds when $x$ is of the form $x = \sum_{g\in S} \, x_g\, \lambda_\sigma(g)$, where $S$ is a finite subset of $G$ and $x_g \in \C$ for all $g\in S$. We then have
\begin{align*}
\sum_{i=1}^n \, \lambda_\sigma(g_i) \,E\big(\lambda_\sigma(g_i)^*\, x\big) &=
\sum_{i=1}^n \, \lambda_\sigma(g_i) \,\sum_{g\in S} \, x_g\, E\big(\lambda_\sigma(g_i)^*\,\lambda_\sigma(g) \big)\\
& = \sum_{i=1}^n \, \lambda_\sigma(g_i) \,\sum_{g\in S} \, x_g\, \overline{\sigma(g_i^{-1}, g_i)}\, \sigma(g_i^{-1}, g)\, E\big(\lambda_\sigma(g_i^{-1}\, g) \big)\\
& = \sum_{i=1}^n \, \lambda_\sigma(g_i) \,\sum_{g'\in\, g_i^{-1}S} \, x_{g_i g'}\, \overline{\sigma(g_i^{-1}, g_i)}\, \sigma(g_i^{-1}, g_i \,g')\, E\big(\lambda_\sigma(g') \big)\\
& = \sum_{i=1}^n \, \lambda_\sigma(g_i) \,\sum_{h\in\, H\,\cap\, g_i^{-1}S} \, x_{g_i h}\, \overline{\sigma(g_i^{-1}, g_i)}\, \sigma(g_i^{-1}, g_i \,h)\, \lambda_\sigma(h) \\
& = \sum_{i=1}^n \,\sum_{h\in\, H\,\cap\, g_i^{-1}S} \, x_{g_i h}\, \overline{\sigma(g_i^{-1}, g_i)}\, \sigma(g_i^{-1}, g_i \,h)\, \sigma(g_i, h)\, \lambda_\sigma(g_i h) \\
& = \sum_{i=1}^n \,\sum_{h\in\, H\,\cap\, g_i^{-1}S} \, x_{g_i h}\, \lambda_\sigma(g_i h) 
= \sum_{i=1}^n \,\sum_{g\in\, g_iH\,\cap\,S} \, x_{g}\, \lambda_\sigma(g)\\
&= \sum_{g\in S} \, x_g\, \lambda_\sigma(g) = x\,, 
\end{align*}
where we have used that $ \sigma(g_i^{-1}, g_i) \, = \, \sigma(g_i^{-1}, g_i) \, \sigma(e, h) = \sigma(g_i^{-1}, g_i \,h)\, \sigma(g_i,h)$. The proof that $\sum_{i=1}^n\, E\big(x\, \lambda_\sigma(g_i)\big) \, \lambda_\sigma(g_i)^* =\, x$ is similar.

\medskip

It follows from \cite[Proposition~2.1.5]{Wat} that $E$ is of finite index in the sense of Pimsner-Popa, and, moreover, that the extra assumption in part 1.~of \cite[Corollary~4.6]{Popa} is also satisfied. Hence, we may apply part~1.\ and part~2.\ of \cite[Corollary~4.6]{Popa} to conclude that the desired equivalences \eqref{TEI1} and \eqref{TEI2} hold.
\end{proof}

\subsection{Direct limits of groups}
The following result is useful when considering direct limit of groups.
\begin{proposition}
Assume that $G$ is an inductive limit of a directed family of subgroups $\{G_i\}_{i\in I}$. Let $\sigma \in Z^2(G,\T)$ and let $\sigma_i$ denote the restriction of $\sigma$ to $G_i\times G_i$ for each $i\in I$.
Then the following assertions hold:
\begin{itemize}
\item[(i)] If $(G_i,\sigma_i)$ satisfies Kleppner's condition for all $i$, then $(G,\sigma)$ satisfies Kleppner's condition.
\item[(ii)] If $(G_i,\sigma_i)$ is $C^*$-simple for all $i$, then $(G,\sigma)$ is $C^*$-simple.
\item[(iii)] If $(G_i,\sigma_i)$ has the unique trace property for all $i$, then $(G,\sigma)$ has the unique trace property.
\end{itemize}
\end{proposition}

\begin{proof}
If $g$ is a nontrivial $\sigma$-regular element in $G$ with finite conjugacy class,
then there is some $i\in I$ such that $g \in G_i$.
It is easy to check that $g$ is then $\sigma_i$-regular in $G_i$,
and that its conjugacy class in $G_i$ is finite.
Hence, (i) holds. Assertion (ii) and (iii) are consequences of general facts valid for C$^*$-algebras,
for example mentioned in \cite[Proposition~10]{BdH}.
\end{proof}

\subsection{Direct products of groups}

We consider a couple of examples involving direct product of groups.
The first one just says that it is easy to handle product cocycles.
The second one illustrates that other types of cocycles require more work.
\begin{proposition}
For $i=1,2$, let $G_i$ be a group and $\sigma_i \in Z^2(G_i,\T)$.
Set $G=G_1 \times G_2$ and $\sigma = \sigma_1 \times \sigma_2$.
Then it is well known that $C^*_r(G,\sigma) \simeq C^*_r(G_1,\sigma_1) \otimes_{\textup{min}} C^*_r(G_2,\sigma_2)$ and the following statements are easily checked:
\begin{itemize}
\item[(i)] $(G,\sigma)$ satisfies Kleppner's condition if and only if both $(G_1,\sigma_1)$ and $(G_2,\sigma_2)$ satisfy Kleppner's condition.
\item[(ii)] $(G,\sigma)$ is $C^*$-simple if and only if both $(G_1,\sigma_1)$ and $(G_2,\sigma_2)$ are $C^*$-simple.
\item[(iii)] $(G,\sigma)$ has the unique trace property if and only if both $(G_1,\sigma_1)$ and $(G_2,\sigma_2)$ have the unique trace property.
\end{itemize}
\end{proposition}

Note that, in general, if $G=G_1 \times G_2$, $\sigma \in Z^2(G,\T)$, and $\sigma_i$ denotes the restriction of $\sigma$ to $G_i\times G_i$ for $i=1,2$, then none of the above equivalences need to hold, as one can verify by considering various cocycles on $\Z^4 = \Z^2 \times \Z^2$. (Statement (i) is discussed in \cite[Section~3]{Om}).

\begin{example}\label{F2xZ}
Consider the group $G=\F_2 \times \Z$, where $\F_2$ denotes the free group on two generators, say $a$ and $b$.
Clearly, $G$ is non-amenable, hence not FC-hypercentral, and non-ICC.
Nevertheless, $G$ belongs to $\mathcal{K}$.
	
Indeed, as explained in \cite[Example~3.11]{Om}, every $\sigma \in Z^2(G,\T)$ is, up to similarity,
given by $\sigma((x,m),(y,n))=\phi(y,m)$ for some bihomomorphism $\phi\colon \F_2 \times \Z \to \T$.
Letting $\gamma\colon\F_2\to \T$ denote the homomorphism (character) given by $\gamma(x) = \phi(x,1)$, we have $\phi(x,m) = \gamma^m(x)$.
Moreover, $\phi$ is completely determined by $\mu=\gamma(a)$ and $\nu=\gamma(b)$.
The following conditions are then equivalent:
\begin{itemize}
\item[(i)] at least one of $\mu$ and $\nu$ is nontorsion,
\item[(ii)] $(G,\sigma)$ satisfies Kleppner's condition,
\item[(iii)] $(G,\sigma)$ is $C^*$-simple,
\item[(iv)] $(G,\sigma)$ has the unique trace property.
\end{itemize}
The equivalence of (i) and (ii) is shown in \cite[Example~3.11]{Om}.
Next, consider $H=\F_2\times\{0\}$ and let $s\colon \Z=G/H \to G$ be the section given by $s(k)= (e, k)$. 
From subsection~\ref{subgroups} we obtain the crossed product decomposition
\[
C^*_r(G,\sigma) \, \simeq \, C_r^*\big(C^*_r(\F_2), \Z, \beta\big)\, ,
\]
where the action $\beta$ of $\Z$ on $C^*_r(\F_2)$ is untwisted and determined by $\beta_k(\lambda(x))= \gamma^k(x)\,\lambda(x)$ for $x\in\F_2$ and $k\in\Z$.

Assume now that (i) holds. Then the map $m \mapsto \beta_m$ gives an embedding of $\Z$ into $\aut(C^*_r(\F_2))$.
As $\F_2$ is $C^*$-simple and has the unique trace property, we can then use \cite[Theorem~7]{Bed2} to conclude that both (iii) and (iv) hold. Alternatively, we could have used \cite{Yin} here.
Finally, as pointed out before, the implications (iii)~$\Rightarrow$~(ii) and (iv)~$\Rightarrow$~(ii) always hold.
\end{example}

\subsection{More on FC-hypercentral groups}\label{ICCG}

Set $ICC(G):= G/FCH(G)$.
We first remark that $ICC(G)$ has the unique trace property, i.e., $ICC(G)$ has trivial amenable radical, if and only if $FCH(G)=AR(G)$.

Indeed, if $FCH(G)=AR(G)$, then $ICC(G) = G/AR(G)$, which has trivial amenable radical.
The converse implication follows from the fact that if $N$ is a normal subgroup of $G$ such that $G/N$ has the unique trace property,
then $AR(G)\subset N$ (see \cite[Lemma~6.11]{Iva-Om}, and the comment before it).

In the same way, it can be shown that $ICC(G)$ is $C^*$-simple if and only if $FCH(G)=AH(G)$, where $AH(G)$ denote the \emph{amenablish radical} of $G$, as introduced in \cite{Iva-Om}.

\begin{theorem}\label{FCH-P}
Assume that $FCH(G) = AR(G)$, or equivalently, that $ICC(G)$ has the unique trace property.
Then $(G,\sigma)$ has the unique trace property whenever $(G,\sigma)$ satisfies Kleppner's condition.
Hence, $G$ belongs to $\K_{UT}$.
\end{theorem}
\begin{proof}
Suppose that $(G,\sigma)$ satisfies Kleppner's condition. Set $H= FCH(G)$ and $K = ICC(G)$.
Applying \cite[Proposition~4.3]{BO}, we get that the canonical tracial state on $C^*_r(H, \sigma')$ is the only $K$-invariant tracial state on $C^*_r(H, \sigma')$.
Since $K$ has the unique trace property, it follows from Proposition~\ref{BryKed}~(ii) (i.e., from \cite[Corollary~5.3]{BK}) that $(G,\sigma)$ has the unique trace property.
\end{proof}

\begin{remark} \label{ICCsimple}
Let us consider the case where $ICC(G)$ is $C^*$-simple. Then $ICC(G)$ has the unique trace property, so Theorem~\ref{FCH-P} gives that $G$ lies in $\K_{UT}$, and one may wonder whether it will always lie in $\K$. Set $H=FCH(G)$. The problem is then to decide if $K=ICC(G)$ acts on $C_r^*(H, \sigma')$ in a minimal way when $\sigma \in K(G)$, since Proposition~\ref{BryKed}~(i) will then imply that $(G,\sigma)$ is $C^*$-simple.
\end{remark}	
An example of a situation where $ICC(G)$ acts on $C_r^*(FCH(G),\sigma')$ in a minimal way is when $(FCH(G),\sigma')$ satisfies Kleppner's condition,
because $FCH(G)$ is FC-hypercentral, so it follows from Theorem~\ref{FCH} that $C_r^*(FCH(G),\sigma')$ is simple in this case.
Hence 
Proposition~\ref{BryKed}~(i) and Theorem~\ref{FCH-P} give:
\begin{corollary}
If $ICC(G)$ is $C^*$-simple and $(FCH(G),\sigma')$ satisfies Kleppner's condition, then $(G, \sigma)$ is $C^*$-simple with the unique trace property.
\end{corollary}

\begin{example}\label{F2xZ-2}
The procedure described in Remark~\ref{ICCsimple} works well when $G=\F_2 \times \Z$, as in Example~\ref{F2xZ}.
It is not difficult to check that $H=FCH(G)= \{e\}\times \Z \simeq \Z$, so $K=ICC(G) \simeq \F_2=\langle a, b\rangle$, which is $C^*$-simple.
Let $\sigma \in Z^2(G,\T)$ be determined by $\mu$ and $\nu$ in $\T$ as in Example~\ref{F2xZ}.
Then $\sigma'=1$, so $C^*_r(H,\sigma') = C_r^*(\Z)$.
Moreover, choosing the section $s\colon K\to G$ given by $s(x) = (x,0)$, we get from subsection~\ref{subgroups} that
\[
C^*_r(G,\sigma) \simeq C_r^*(C_r^*(\Z), \F_2, \beta)\, ,
\]
where the action $\beta$ of $\F_2$ on $C_r^*(\Z)$ is untwisted and determined by
\[
\beta_x(\lambda(m)) = \overline{\mu}^{\ m\,o_a(x)}\,\overline{\nu}^{\ m\,o_b(x)}\,\lambda(m)
\]
for $x \in \F_2$ and $m \in \Z$,
where $o_a$ (resp.~$o_b$)$\colon\F_2 \to \Z$ denotes the homomorphism sending $a$ to $1$ and $b$ to $0$ (resp.\ sending $a$ to $0$ and $b$ to $1$).
Identifying $C_r^*(\Z)$ with $C(\T)$ via the Gelfand transform, we get that each $\beta_x$ is the $*$-automorphism of $C(\T)$ associated to the homeomorphism $\varphi_x$ of $\T$ given by
\[
\varphi_x (z) = \mu^{\,o_a(x)}\,\nu^{\,o_b(x)}\,z
\]
for $z\in \T$.
Hence, if at least one of $\mu$ and $\nu$ is nontorsion, we see that every orbit $\{\varphi_x(z) : x\in \F_2\}$ is dense in $\T$, so the action of $\F_2$ on $C^*_r(H,\sigma')=C_r^*(\Z)$ is minimal. We can therefore conclude that $(G, \sigma)$ is $C^*$-simple and has the unique trace property in this case, in accordance with what we found in Example~\ref{F2xZ}.
\end{example}

The next example shows that the class of solvable groups is not contained in $\K$,
and that the class of groups with exponential growth is neither contained in $\K_{C^*S}$ nor in $\K_{UT}$.
It also gives an example of an amenable ICC group $G$ satisfying $\emptyset \neq C^*S(G)=UT(G) \neq K(G)$.

\begin{example}\label{anosov}
$C^*$-simplicity of $(G,\sigma)$ when $G$ is a semidirect product of the form $\Z^n \rtimes_A \Z$ for some $A \in\op{GL}(n,\Z)$ is thoroughly discussed by Packer and Raeburn in \cite[Theorem~3.2]{PR} (see also subsection~\ref{aperiodic-autos} below, in particular Example~\ref{Zn-by-Z}). To make our point, it will suffice to consider a matrix
\[
A=\begin{bmatrix}a&b\\c&d\end{bmatrix}\in\op{GL}(2,\Z),
\]
and the action of $\Z$ on $\Z^2$ associated with $A$, that is,
\[
k \cdot \mathbf{x} = A^k \mathbf{x}
\]
for $k \in \Z$ and $\mathbf{x} \in \Z^2$.
Let $G = \Z^2 \rtimes_A \Z$ denote the corresponding semidirect product, which is clearly a solvable group.
Computations show that $G$ is ICC (and has exponential growth) if and only if $\lvert a+d \rvert > 1 + \op{det}A$.
This holds for example when $a=2$ and $b=c=d=1$.
Assuming this, and making use of \cite[Example~3.4]{PR}, we have that any $\sigma \in Z^2(G,\T)$ is similar to $\check\sigma_\theta$ for some $\theta \in [0, 1/2)$, where
\[
\check\sigma_\theta\big((\mathbf{x},k),(\mathbf{y},l)\big)=\exp 2\pi i \, \Big(\mathbf{x}^t\begin{bmatrix}0&\theta\\-\theta&0\end{bmatrix} A^k\mathbf{y}\Big)
\]
for $\mathbf{x}, \mathbf{y} \in \Z^2$ and $k, l \in \Z$.
Moreover,
\begin{equation}\label{theta-dec}
C^*_r(G,\sigma) \simeq C^*_r(\Z^2\rtimes_A\Z, \check\sigma_\theta) \simeq C_r^*\big(C^*_r(\Z^2,\sigma_\theta),\Z, \beta\big)\,,
\end{equation}
where \[
\sigma_\theta\big(\mathbf{x},\mathbf{y}\big)=\exp 2\pi i \, \Big(\mathbf{x}^t\begin{bmatrix}0&\theta\\-\theta&0\end{bmatrix} \mathbf{y}\Big)\,,
\]
the action $\beta\colon\Z\to \aut(C^*_r(\Z^2,\sigma_\theta))$ being determined by $\beta_k\big(\lambda_{\sigma_\theta}(\mathbf{x})\big) = \lambda_{\sigma_\theta}(A^k\mathbf{x})$ for $\mathbf{x}, \mathbf{y} \in \Z^2$ and $k \in \Z$.

Consider now the statements
\begin{itemize}
\item[(i)] $\theta$ is irrational,
\item[(ii)] $(G,\check\sigma_\theta)$ is $C^*$-simple,
\item[(iii)]$(G,\check\sigma_\theta)$ has the unique trace property.
\end{itemize}
Then these three statements are equivalent.
Indeed, (ii)~$\Rightarrow$~(i) follows by applying \cite[Theorem~3.2]{PR}. Using the decomposition \eqref{theta-dec}, one sees that the implication (i)~$\Rightarrow$~(iii) is a special case of \cite[Theorem~8]{Bed2} (and its proof). Finally, the implication (iii)~$\Rightarrow$~(ii) follows from Theorem~\ref{Murph} since $G$ is amenable.

However, as $G$ is ICC, $(G,\check\sigma_\theta)$ always satisfies Kleppner's condition, also when $\theta$ is rational.
So we see that $G$ does not belong to $\K_{C^*S}$, nor to $\K_{UT}$.
\end{example}

To deal with similar situations, the following somewhat curious notion may turn out to be useful.
Let us say that $(G,\sigma)$ satisfies \emph{condition~X} if there exists a normal subgroup $N$ of $G$ such that
\begin{itemize}
\item[(i)] $FC(G) \subset N$,
\item[(ii)] $G/N$ is FC-hypercentral,
\item[(iii)] for all $h \in N\setminus\{e\}$, there exists $g \in G$ such that $hg=gh$ and $\sigma(h,g) \neq \sigma(g,h)$.
\end{itemize}
Note that $FC(G) \subset FC(N)$ if and only if $FC(G) \subset N$.

\medskip

In general, condition~X implies Kleppner's condition, as can be seen by combining (i) and (iii).
Moreover, if $G$ is FC-hypercentral, then $(G,\sigma)$ satisfies condition~X if and only $(G,\sigma)$ satisfies Kleppner's condition. Indeed, if Kleppner's condition hold, then we may take $N=FC(G)$ to see that condition~X holds.

\begin{proposition}
Let $G$ be an amenable group and assume that $(G,\sigma)$ satisfies condition~X.
Then $(G,\sigma)$ is $C^*$-simple and has the unique trace property.
\end{proposition}

\begin{proof}
The result is a generalization of \cite[Theorem~3.1]{BO}.
Instead of using $FC(G)$ as the ``base case'' in the inductive proof of this theorem,
we replace it by the (larger) normal subgroup $N$.
Then the same proof as in \cite{BO} will work,
provided that $G/N$ is FC-hypercentral and $N$ (and thus $G$) is amenable. We leave the details to the reader.
\end{proof}

This proposition seems potentially applicable when dealing with solvable groups and ``FC-hypercentral-by-FC-hypercentral'' groups.
For example, it may used it to show that (i) implies (ii) and (iii) in Example~\ref{anosov}:
choosing $N=\Z^2$, one readily checks that $(\Z^2\rtimes_A \Z, \check\sigma_\theta)$ satisfies condition~X whenever $\theta$ is irrational.

\section{On normal subgroups and freely acting automorphisms}\label{free}

Throughout this section, we assume that $H$ is a normal subgroup of $G$ and set $K = G/H$.
As before, the restriction of $\sigma \in Z^2(G,\T)$ to $H\times H$ will be denoted by $\sigma'$, and $\tau'$ will denote the canonical tracial state on $W^*(H, \sigma')$ (resp.~$C_r^*(H, \sigma')$). We recall from subsection~\ref{subgroups} that for each $g\in G$ there exists $\gamma_g\in \aut(W^*(H, \sigma'))$ satisfying
\[
\gamma_g\big(\lambda_{\sigma'}(h)\big)=\widetilde\sigma(g,h)\,\lambda_{\sigma'}(g\cdot h) \quad \text{ for all } h \in H\,.
\] We fix a section $s\colon K\to G$ for the canonical homomorphism $q$ from $G$ onto $K$ satisfying $s(e)=e$, and let $(\beta, \omega)$ denote the associated twisted action of $K$ on $W^*(H, \sigma'))$.
We otherwise freely use the notation introduced in subsection~\ref{subgroups}. 

Our main goal in this section is to provide a set of conditions on $G, H$ and $\sigma$ guaranteeing that $(G,\sigma)$ has the unique trace property, or is $C^*$-simple with the unique trace property (see Theorem~\ref{utp}). For the unique trace property, our plan is to invoke \cite[Proposition~9]{Bed4}, and our first task will therefore be to find a condition ensuring that $\gamma_k \in \aut(W^*(H,\sigma'))$ is freely acting
in the sense of Kallman \cite{Kal} (see also \cite{Str}) for each $k \in G\setminus H$.
 We will next show that $C^*$-simplicity may then be deduced in certain cases from various results, e.g.\ (the twisted version of) Kishimoto's theorem \cite[Theorem~3.1]{Kis}. 
 
 For the convenience of the reader, we recall that if $M$ is a von~Neumann algebra and $\alpha \in \aut(M)$, then $\alpha$ is called \emph{freely acting} (or~\emph{properly outer}) if the only element $T\in M$ satisfying $\alpha(S)T=TS$ for all $S \in M$ is $T=0$. Equivalently, $\alpha$ is freely acting if the restriction $\alpha_{|Mp}$ is outer for every nonzero central projection $p$ in $M$ satisfying $\alpha(p)=p$. We also recall that a twisted action $(\beta, \omega)$ of a group $K$ on $M$ is called \emph{freely acting} (or~\emph{properly outer}) if $\beta_k$ is freely acting for every $k \in K\setminus\{e\}$.

\begin{lemma}\label{k-central}
Let $T \in W^*(H,\sigma')$ and $k \in G$. Define $f_T \in \ell^2(H)$ by $f_T=T\delta_e$.
Then the following conditions are equivalent:
\begin{itemize}
\item[(i)] \, $\gamma_k(S)T=TS$ for all $S \in W^*(H,\sigma')$.
\item[(ii)] \, $\widetilde\sigma(k,s)\,\overline{\sigma(t,s)}\,\sigma\big(k\cdot s,(k\cdot s)^{-1}ts\big)\,f_T\big((k \cdot s)^{-1}ts\big)=f_T(t)$ for all $s,t \in H$.
\end{itemize}
\end{lemma}
\begin{proof}
Since $W^*(H,\sigma')=\lambda_{\sigma'}(H)''$, it is clear that (i) holds if and only if
\[
\gamma_k\big(\lambda_{\sigma'}(s)\big)\,T=T\,\lambda_{\sigma'}(s) \quad \text{ for all } s \in H.
\]
Hence, since $\delta_e$ is a separating vector for $W^*(H,\sigma')$ and
\[
\lambda_{\sigma'}(s)\rho_{\overline{\sigma'}}(s)\delta_e=\rho_{\overline{\sigma'}}(s)\lambda_{\sigma'}(s)\delta_e=\delta_e \quad \text{ for all } s \in H,
\]
(i) is equivalent with
\begin{equation}\label{freelyact}
\gamma_k(\lambda_{\sigma'}(s))\, \rho_{\overline{{\sigma'}}}(s)\,T\delta_e=T\delta_e \quad \text{ for all } s \in H.
\end{equation}
Let $t \in H$. Evaluating the left hand side of equation \eqref{freelyact} at $t$ gives
\[
\begin{split}
& \Big(\gamma_k\big(\lambda_{\sigma'}(s)\big)\,\rho_{\overline{{\sigma'}}}(s)\,f_T\Big)(t) \\
&\quad = \Big(\widetilde\sigma(k,s)\, \lambda_{\sigma'}(k\cdot s)\, \rho_{\overline{{\sigma'}}}(s)\,f_T\Big)(t) \\
&\quad = \widetilde\sigma(k,s)\, \sigma\big(k\cdot s, (k\cdot s)^{-1}t\big)\,\Big(\rho_{\overline{{\sigma'}}}(s)\, f_T\Big)\big((k\cdot s)^{-1}\,t\big) \\
&\quad = \widetilde\sigma(k,s)\,\sigma\big(k\cdot s,(k\cdot s)^{-1}t\big)\,\overline{\sigma((k\cdot s)^{-1}t,s)}\, f_T\big((k\cdot s)^{-1}ts\big)\,,
\end{split}
\]
and (i) is now seen to be equivalent to (ii) by making use of \eqref{coceq}.

\end{proof}

Let $g \in G$. We let $C_H(g)$ denote the \emph{$H$-conjugacy class of $g$ in $G$}, that is,
\[
C_H(g) = \{ sgs^{-1} : s \in H\}.
\]
Moreover, if $k \in G$, we define the \emph{$(k,H)$-conjugacy class of $g$ in $G$} by
\[
C_H^{\,k}(g)=\{ (k\cdot s) \, g \, s^{-1} : s \in H\}.
\]
This class is nothing but the equivalence class of $g$ w.r.t.\ the equivalence relation on $G$ defined by $g' \sim_k g$ whenever $g'=(k\cdot s) \, g \, s^{-1}$ for some $ s \in H$. Clearly, we have $C_H^k(g) \subset H$ if and only if $g \in H$.
\medskip
We note that $C_H^{\,k}(g) = k \, C_H(k^{-1}g)$. This gives
\begin{equation*}
\lvert C_H^{\,k}(g) \rvert = \lvert C_H(k^{-1}g) \rvert \leq \lvert C_G(k^{-1}g) \rvert\,.
\end{equation*}

We will also need the following definitions:

\begin{definition}
Let $g \in G$. We say that $g$ is \emph{$\sigma$-regular w.r.t.\ $H$} if \[\sigma(g,s)=\sigma(s,g)\] whenever $s \in H$ commutes with $g$.
\end{definition}
\begin{definition}
Let $t \in H$ and $k \in G$. We say that $t$ is \emph{$\sigma$-regular w.r.t.\ $(k,H)$} if \[\sigma(k^{-1}t,s)=\sigma(s,k^{-1}t)\] whenever $s \in H$ and $k^{-1}t s = s k^{-1}t$ (that is, $(k\cdot s)t = ts$).
\end{definition}

Clearly, for $k \in G$ and $t \in H$, we have
\[
\text{$k$ is $\sigma$-regular w.r.t.\ $G$ } \Longrightarrow \text{ $k$ is $\sigma$-regular w.r.t.\ $H$}
\]
and
\[
\text{$k^{-1}t$ is $\sigma$-regular w.r.t.\ $H$ } \Longleftrightarrow \text{ $t$ is $\sigma$-regular w.r.t.\ $(k,H)$}.
\]

\begin{lemma}\label{reg-class}
The following hold:
\begin{itemize}\itemsep4pt
\item[(i)] Let $x \in G$ and $y \in C_H(x)$. \\
If $x$ is $\sigma$-regular w.r.t.\ $H$, then $y$ is $\sigma$-regular w.r.t.\ $H$.
\item[(ii)] Let $k \in G$\,, $t \in H$ and $t' \in C_H^{\,k}(t)$. \\
If $t$ is $\sigma$-regular w.r.t.\ $(k,H)$, then $t' $ is $\sigma$-regular w.r.t.\ $(k,H)$.
\end{itemize}
\end{lemma}

\begin{proof}
(i) Assume that $x$ is $\sigma$-regular w.r.t.\ $H$. Write $y=rxr^{-1}$ for some $r \in H$, and assume $ys=sy$ for some $s \in H$.
We have to show that $\sigma(y,s)=\sigma(s,y)$.

\medskip
Using the cocycle identity \eqref{coceq} twice, one readily checks that
\[
\sigma(s,y)\overline{\sigma(y,s)} = \sigma(y,r) \sigma(s,rx) \overline{\sigma(y,sr)}\,\overline{\sigma(s,r)}.
\]
Now, as $xr^{-1}sr=r^{-1}srx$ and $r^{-1}sr \in H$, the $\sigma$-regularity of $x$ w.r.t.\ $H$ gives that $\sigma(x,r^{-1}sr)=\sigma(r^{-1}sr,x)$. Using this, some further cocycle computations give that
\[
\sigma(y,sr)=\overline{\sigma(rx,r^{-1})}\, \overline{\sigma(r,x)}\, \sigma(r,r^{-1}sr)\sigma(sr,x)\sigma(r^{-1},sr).
\]
Thus, we get
\[
\begin{split}
& \sigma(s,y)\overline{\sigma(y,s)} \\
& \quad =\sigma(y,r) \sigma(s,rx)
\sigma(rx,r^{-1})\sigma(r,x)\overline{\sigma(r,r^{-1}sr)}\, \overline{\sigma(sr,x)}\, \overline{\sigma(r^{-1},sr)}\,\overline{\sigma(s,r)}\\
& \quad =
\overline{\sigma(s,r)}\, \overline{\sigma(sr,x)}\sigma(s,rx)\sigma(r,x)
\cdot \sigma(rx,r^{-1})\sigma(y,r)
\cdot \overline{\sigma(r,r^{-1}sr)} \,\overline{\sigma(r^{-1},sr)} \\
& \quad = 1 \cdot \sigma(r^{-1},r) \cdot \overline{\sigma(r,r^{-1})} = 1\,.
\end{split}
\] 

(ii) Assume $t$ is $\sigma$-regular w.r.t.\ $(k,H)$.
Then $x:=k^{-1}t$ is $\sigma$-regular w.r.t.\ $H$ and $t' = k y$ for some $y \in C_H(x)$. So (i) gives that $y$ is $\sigma$-regular w.r.t.\ $H$. Hence $t' = k y$ is $\sigma$-regular w.r.t.\ $(k,H)$, as desired.
\end{proof}

Lemma~\ref{reg-class} shows that if some $H$-conjugacy class contains an element which is $\sigma$-regular w.r.t.\ $H$, then all its elements are also $\sigma$-regular w.r.t.\ $H$; we will therefore call such a $H$-conjugacy class for $\sigma$-regular.

This lemma also shows that if some $(k,H)$-conjugacy class in $H$ contains an element which is $\sigma$-regular w.r.t.\ $(k,H)$, then all its elements are also $\sigma$-regular w.r.t.\ $(k,H)$; we will therefore say that such a $(k,H)$-conjugacy class in $H$ is $\sigma$-regular.

\begin{definition}
The triple $(G,H,\sigma)$ is said to satisfy \emph{the relative Kleppner condition} if,
for every $k \in G\setminus H$, all $\sigma$-regular $(k,H)$-conjugacy classes in $H$ are infinite, that is, we have:

\begin{itemize}
\item[(1)] $\lvert C_H^{\,k}(t) \rvert = \infty$ whenever $k \in G \setminus H$, $t \in H$ and $C_H^{\,k}(t)$ is $\sigma$-regular.
\end{itemize}
As is easily checked, this is equivalent to:
\begin{itemize}
\item[(2)] $\lvert C_H(g) \rvert = \infty$ whenever $g \in G \setminus H$ and $C_H(g)$ is $\sigma$-regular.
\end{itemize}
\end{definition}

\begin{remark}\label{relK}
\item[a)] If $H=G$, then the relative Kleppner condition holds trivially. In the opposite direction,
if $H=\{e\}$, then the relative Kleppner condition never holds, as immediately follows from (2).
\smallskip
\item[b)] $(G,H, 1)$ satisfies the relative Kleppner condition if and only if $\lvert C_H^{\, k}(t) \rvert = \infty$ whenever $k \in G \setminus H$ and $t \in H$, if and only if $\lvert C_H(g) \rvert = \infty$ whenever $g \in G \setminus H$. In particular, it follows that $(G,H, \sigma)$ satisfies the relative Kleppner condition whenever $(G,H, 1)$ satisfies the relative Kleppner condition.
\smallskip
\item[c)] Assume that $C_H(g)$ is finite for all $g \in G \setminus H$. For instance, this holds when $H$ is central or finite. Then $(G,H,\sigma)$ satisfies the relative Kleppner condition if and only if there does not exist any $\sigma$-regular element in $G\setminus H$.
\smallskip
\item[d)] Suppose that $(G,H,\sigma)$ satisfies the relative Kleppner condition and that $H'$ is a normal subgroup of $G$ containing $H$.
Then $(G,H',\sigma)$ satisfies the relative Kleppner condition.

Indeed, let $g \in G \setminus H' \subset G \setminus H$ and suppose $\sigma(g,h)=\sigma(h,g)$ whenever $gh=hg$ and $h \in H'$. Then $\sigma(g,h)=\sigma(h,g)$ whenever $gh=hg$ and $h \in H$, so $\lvert C_H(g) \rvert = \infty$. Hence, $\lvert C_{H'}(g) \rvert \geq \lvert C_H(g) \rvert = \infty$.
\smallskip
\item[e)] We have that $(G,\sigma)$ satisfies Kleppner's condition and, at the same time, $(G,H,\sigma)$ satisfies the relative Kleppner condition
if (and only if) the following two conditions hold:
\begin{itemize}
\item[(i)] $\lvert C_G(h) \rvert = \infty$ whenever $h \in H \setminus \{e\}$ and $C_G(h)$ is $\sigma$-regular,

\smallskip \item[(ii)] $\lvert C_H(g)\rvert = \infty$ whenever $g \in G \setminus H$ and $C_H(g)$ is $\sigma$-regular.
\end{itemize}

Indeed, assume that (i) and (ii) hold. In particular, $(G,H,\sigma)$ satisfies the relative Kleppner condition. Consider $g\in G\setminus H$ such that $C_G(g)$ is $\sigma$-regular. Then $C_H(g)$ is $\sigma$-regular. Thus, using (ii), we get $\lvert C_G(g)\rvert \geq \lvert C_H(g)\rvert = \infty$. Together with (i), this shows that $(G,\sigma)$ satisfies Kleppner's condition. (The converse assertion is trivial).
\end{remark}

\begin{proposition}\label{freeact}
Assume that $(G,H,\sigma)$ satisfies the relative Kleppner condition.
Then $\gamma_k$ is freely acting for every $k \in G \setminus H$.
Moreover, the twisted action $(\beta, \omega)$ of $K$ on $W^*(H, \sigma')$ is freely acting.
\end{proposition}

\begin{proof}
Let $k \in G\setminus H$ and suppose $T\in W^*(H,\sigma')$ satisfies $\gamma_k(T)S = ST $ for all $S\in W^*(H,\sigma')$. Using (ii) from Lemma~\ref{k-central}, we get that
\[
\lvert f_T\rvert \big((k\cdot s) t s^{-1}\big) = \lvert f_T\rvert(t)
\]
for all $s, t \in H$. This means that $\lvert f_T\rvert$ is constant on each $(k,H)$-conjugacy class $C_H^k(t)$.

Let $t \in H$. Assume first that $C_H^k(t)$ is $\sigma$-regular. Since $(G,H,\sigma)$ satisfies the relative Kleppner condition, we have $\lvert C_H^k(t)\rvert=\infty $. As $f_T \in \ell^2(H)$, we get that $\lvert f_T\rvert$ is constantly equal to zero on $C_H^k(t)$. Hence, $f_T=0$ on $C_H^k(t)$.

Assume now that $C_H^k(t)$ is not $\sigma$-regular. So there exists $s \in H$ such that
\begin{equation}\label{sig-eq1}
(k\cdot s)t = ts
\end{equation}
and
\begin{equation}\label{sig-eq2}
\overline{\sigma(k^{-1}t, s)} \sigma(s, k^{-1}t) \neq 1\,.
\end{equation}
Using equation \eqref{sig-eq1} and (ii) in Lemma~\ref{k-central}, we get
\begin{equation} \label{sigma-eq1}
\widetilde\sigma(k,s)\,\overline{\sigma(t,s)}\,\sigma\big(k\cdot s,t\big)\,f_T(t)=f_T(t)\,.
\end{equation}
Some detailed but routine cocycle computations give that
\[
\widetilde\sigma(k,s)\,\overline{\sigma(t,s)}\,\sigma\big(k\cdot s,t\big) = \overline{\sigma(k^{-1}t,s)}\sigma(s,k^{-1}t)\,.
\]
Thus, using \eqref{sig-eq2}, we get 
\[
\widetilde\sigma(k,s)\,\overline{\sigma(t,s)}\,\sigma\big(k\cdot s,t\big) \neq 1\,,
\]
so we conclude from \eqref{sigma-eq1} that $f_T(t)=0$. As $\lvert f_T\rvert$ is constant on $C_H^k(t)$, we get that $f_T=0$ on $C_H^k(t)$.

Altogether, we have shown that $f_T=0$ on each $(k,H)$-conjugacy class in $H$. Since $H$ is the union of all such classes, it follows that $f_T=0$ on the whole of $H$. As $\delta_e$ is separating for $W^*(H,\sigma')$, we get that $T=0$. This proves that $\gamma_k$ is freely acting, as desired.

Finally, recall that $\beta_k = \gamma_{s(k)}$ for each $k\in K$, where $s\colon K\to G$ denotes the chosen section for the quotient map from $G$ onto $K$. Since $s(k) \in G\setminus H$ for every $k \in K\setminus\{e\}$, it follows that $(\beta, \omega)$ is freely acting.
\end{proof}

\begin{remark}
It can be shown that 
if $\gamma_k$ is freely acting for every $k \in G \setminus H$, then $(G,H,\sigma)$ satisfies the relative Kleppner condition.
As we will not need this fact, we leave this as an exercise for the reader.
\end{remark}

\begin{theorem}\label{utp}
Assume that $(G,H,\sigma)$ satisfies the relative Kleppner condition and that $\tau'$ is the unique $K$-invariant tracial state of $C_r^*(H, \sigma')$. Then $(G,\sigma)$ has the unique trace property.

Assume, in addition, that at least one of the following two conditions is satisfied:
\begin{itemize}
\item[(a)] $G$ is amenable,
\item[(b)] $G$ is exact and $C^*_r(G,\sigma)$ has stable rank one.
\end{itemize}
Then $(G,\sigma)$ is $C^*$-simple.
\end{theorem}

\begin{proof}
Set $A=C_r^*(H, \sigma')$. We first have to show that $C^*(G,\sigma)\simeq C_r^*\big(A, K, \beta^r,\omega\big)$ has a unique tracial state. Since $\tau'$ is assumed to be the unique $K$-invariant tracial state of $A$, according to \cite[Proposition~9]{Bed4}, it suffices to check that the twisted action $(\beta^r, \omega)$ of $K$ on $A$ is tracially properly outer in the sense of \cite{Bed4}. As the GNS-representation of $C_r^*(H, \sigma')$ associated to $\tau'$ is the identity representation of $A$ on $\ell^2(H)$, this amounts to checking that $(\beta,\omega)$ is freely acting on $A'' = W^*(H,\sigma')$. Since $(G,H,\sigma)$ is assumed to satisfy the relative Kleppner condition, this follows from Proposition~\ref{freeact}.

If (a) or (b) also holds, then combining the first assertion with Theorem~\ref{Murph} gives that $(G,\sigma)$ is $C^*$-simple.
\end{proof}

\begin{remark}
It follows from \cite[Proposition~15~(i)]{Bed4} (see also \cite[Proposition~6]{BdH}) that if
\begin{equation}\label{free-cond}
\lvert C_H(g)\rvert =\infty\,\, \text{ for all } g \in G\setminus H\,
\end{equation}
and $H$ has the unique trace property, then $G$ has the unique trace property. Since condition \eqref{free-cond} corresponds to the relative Kleppner condition for $(G,H,1)$, the first assertion in Theorem~\ref{utp} provides a twisted version of this result.
\end{remark}

Proposition~\ref{freeact} and Theorem~\ref{utp} have several interesting corollaries.

\begin{corollary}\label{sut-relKlep}
Assume that $(G,H,\sigma)$ satisfies the relative Kleppner property. Then the following assertions hold:
\begin{itemize}
\item[(i)] $(G,\sigma)$ has the unique trace property whenever $(H,\sigma')$ has the unique trace property.
\item[(ii)]
$(G,\sigma)$ is $C^*$-simple whenever $(H,\sigma')$ is $C^*$-simple.
\item[(iii)]
$(G,\sigma)$ is $C^*$-simple with the unique trace property whenever $(H,\sigma')$ is $C^*$-simple with the unique trace property.
\end{itemize}
\end{corollary}
\begin{proof} The first assertion is an immediate consequence of Theorem~\ref{utp}. Next, suppose $(H,\sigma')$ is $C^*$-simple. Then $W^*(H, \sigma')$ is a factor, so it follows from Proposition~\ref{freeact} that the twisted action $(\beta, \omega)$ of $K$ on $W^*(H, \sigma')$ is outer. This implies that the twisted action $(\beta^r, \omega)$ of $K$ on $C_r^*(H, \sigma')$ is also outer. Hence, \cite[Theorem~3.2]{Bed} (the twisted version of \cite[Theorem~3.1]{Kis}) gives that $C^*_r(G,\sigma) \simeq C_r^*(A, K, \beta^r,\omega)$ is simple. This shows that (ii) holds. The third assertion follows readily from (i) and (ii).
\end{proof}

\begin{corollary}\label{sut-relKlep-2}
Assume that $H$ is FC-hypercentral, $(H,\sigma')$ satisfies Kleppner's condition and $(G,H,\sigma)$ satisfies the relative Kleppner condition. Then $(G,\sigma)$ is $C^*$-simple with the unique trace property.
\end{corollary}

\begin{proof}
As the first two assumptions imply that $(H,\sigma')$ is $C^*$-simple with the unique trace property, cf.~Theorem~\ref{FCH}, this follows from Corollary~\ref{sut-relKlep}~(iii).
\end{proof}

\begin{corollary}\label{2Klepp}
Assume that the following three conditions hold:
\begin{itemize}
\item[(i)] $(G,\sigma)$ satisfies Kleppner's condition;
\item[(ii)] $H$ is contained in $FCH(G)$;
\item[(iii)] $(G, H,\sigma)$ satisfies the relative Kleppner condition.
\end{itemize}
Then $(G,\sigma)$ has the unique trace property. If, in addition, $G$ is amenable, or $G$ is exact and $C^*_r(G,\sigma)$ has stable rank one, then $(G,\sigma)$ is $C^*$-simple.
\end{corollary}

\begin{proof}
Using Remark~\ref{relK}~d), it follows from (ii) and (iii) that $(G, FCH(G),\sigma)$ satisfies the relative Kleppner condition.
If we let $\sigma_0$ denote the restriction of $\sigma$ to $FCH(G)\times FCH(G)$, then we get from \cite[Proposition~4.3]{BO} that (i) is equivalent to $C_r^*\big(FCH(G), \sigma_0\big)$ having a unique $ICC(G)$-invariant tracial state. Hence, the result follows from Theorem~\ref{utp}.
\end{proof}

\begin{remark}
To apply Corollary~\ref{2Klepp}, the natural choices for $H$ are $Z(G)$, $FC(G)$, and $FCH(G)$. Remark~\ref{relK}~e) is then useful to check that conditions (i) and (iii) hold, as will be illustrated in the next section.
\end{remark}

Another useful result is:
\begin{corollary}\label{sut-ab}
Assume that $(G,H,\sigma)$ satisfies the relative Kleppner condition and that $\tau'$ is the unique $K$-invariant tracial state of $C_r^*(H, \sigma')$. If $C_r^*(H,\sigma')$ is commutative and $K$ acts on $C_r^*(H,\sigma')$ in a minimal way, then
 $(G,\sigma)$ is $C^*$-simple with the unique trace property.
\end{corollary}
\begin{proof} We know from Theorem~\ref{utp} that the first two assumptions imply that $(G,\sigma)$ has the unique trace property. As seen in the proof of this result, $(\beta^r, \omega)$ is then a tracially properly outer twisted action of $K$ on $A:=C_r^*(H,\sigma')$. Since $A$ is commutative and $K$ acts on $A$ in a minimal way, it follows from \cite[Theorem~10,~part~(b), case~(ii)]{Bed4} that $C^*_r(G,\sigma) \simeq C_r^*(A, K, \beta^r,\omega)$ is simple.
\end{proof}

We also include the following result:
\begin{corollary}\label{sut-torsionfree}
Assume that $(G,H,\sigma)$ satisfies the relative Kleppner condition, $H$ is countable and $K$ is torsion free.
If $K$ acts on $C_r^*(H,\sigma')$ in a minimal way, then $(G,\sigma)$ is $C^*$-simple. Moreover, if, in addition, $\tau'$ is the unique $K$-invariant tracial state of $C_r^*(H, \sigma')$, then $(G,\sigma)$ is $C^*$-simple and has the unique trace property.
\end{corollary}

\begin{proof}
Assume that $K$ acts on $A:=C_r^*(H,\sigma')$ in a minimal way. To show that $C^*_r(G,\sigma) \simeq C_r^*(A, K, \beta^r,\omega)$ is simple, it suffices then to show that for each $k\in K\setminus \{e\}$, $\beta^r_k$ is properly outer as a $*$-automorphism of $A$, as defined in \cite{OP}. Indeed, this follows from \cite[Theorem~7.2]{OP} by noting that $A$ is separable when $H$ is countable and that the proof of Olesen and Pedersen's result is still valid in the case of a twisted action. Now, we know from Proposition~\ref{freeact} that the twisted action $(\beta, \omega)$ of $K$ on $W^*(H,\sigma')$ is freely acting. Using that $K$ is torsion free, we may copy the argument given in the proof of \cite[Theorem~10, part~(b), case~(iii)]{Bed4} to deduce from this fact that $\beta^r_k$ is properly outer for every $k\in K\setminus \{e\}$.

The second assertion follows from the first assertion combined with Theorem~\ref{utp}.
\end{proof}

It is known that if the centralizer $Z_G(H)$ of $H$ in $G$ is trivial and $H$ is $C^*$-simple (resp.\ has the unique trace property), then $G$ is $C^*$-simple (resp.\ has the unique trace property), cf.~\cite{Bed, Bed4}. We can generalize this to the twisted case as follows.

\begin{definition}
The \emph{$\sigma$-centralizer of $H$ in $G$} is the subset of $G$ given by
\[
Z_G^\sigma(H) = \{ g \in G : gs=sg \text{ and } \sigma(g,s)=\sigma(s,g) \text{ for all } s \in H \} \,.
\]
In other words,
\[
Z_G^\sigma(H) = Z_G(H) \, \cap \, \{ g \in G : \text{$g$ is $\sigma$-regular w.r.t.\ $H$}\} \,.
\]
\end{definition}

\begin{proposition}\label{trivial-centralizer}
Assume that $H$ is ICC and $Z_G^\sigma(H)$ is trivial.
If $(H,\sigma')$ is $C^*$-simple \(resp.\ has the unique trace property\),
then $(G,\sigma)$ is $C^*$-simple \(resp.\ has the unique trace property\).
\end{proposition}
 
\begin{proof}
We first prove that $(G,H,\sigma)$ satisfies the relative Kleppner condition.
Assume $g\in G\setminus H$ is $\sigma$-regular w.r.t.\ $H$. We must show that $\lvert C_H(g)\rvert = \infty$.
Suppose that this is not the case. Let $g' \in C_H(g)$, so $g' = sgs^{-1}$ for some $s \in H$.
Then we have $g^{-1}g' = (g^{-1} s g) s^{-1} \in H$.
Moreover, $C_H(g^{-1}g') \subset C_H(g)^{-1} C_H(g') = C_H(g)^{-1} C_H(g)$, so
\[
\lvert C_H(g^{-1}g')\rvert \leq \lvert C_H(g)^{-1} C_H(g')\rvert \leq \lvert C_H(g)\rvert^2 < \infty\,.
\]
Since $H$ is ICC, we must have $g^{-1}g' = e$. Thus, $g'=g$, that is, $C_H(g)=\{g\}$, and it follows that $g\in Z^\sigma_G(H)$.
Since $Z^\sigma_G(H)=\{e\}$, we get that $g=e$, which is impossible since $g\in G\setminus H$.

Since $(G,H,\sigma)$ satisfies the relative Kleppner condition,
Proposition~\ref{freeact} gives that $\beta_k$ is a freely acting automorphism of $W^*(H,\sigma)$ for each $k\in K\setminus \{e\}$.
This implies that $\beta_k^r $ is an outer automorphism of $C_r^*(H, \sigma')$ for each $k\in K\setminus \{e\}$.
Hence, if $(H,\sigma')$ is $C^*$-simple, that is, $A:=C_r^*(H,\sigma')$ is simple,
then it follows from the twisted version of Kishimoto's theorem (see \cite[Theorem~3.2]{Bed})
that $C_r^*(G,\sigma) \simeq C_r^*(A, K, \beta^r, \omega)$ is simple.
On the other hand, if $(H,\sigma')$ has the unique trace property,
then Theorem~\ref{utp} applies and it follows that $(G,\sigma)$ has the unique trace property, as desired.
\end{proof}

\begin{remark}
It is possible that the assumption that $H$ is ICC in Proposition~\ref{trivial-centralizer} is redundant. The proof shows that the argument goes through as long as one knows that $\lvert C_H(g)\rvert \in \{ 1, \infty\}$ for every $g \in G\setminus H$, but we do not see how to deduce this from the assumption that $Z_G^\sigma(H)$ is trivial.
\end{remark}

\begin{remark}
Proposition~\ref{trivial-centralizer} is applied in the study of braid related groups in \cite{Om3}.
There is an action $\alpha$ of the braid group $B_n$ on $n$ strands on the free group $\F_n$, often called ``Artin's representation'', and it shown that the corresponding semidirect product $\F_n\rtimes_\alpha B_n$ belongs to the class $\K$ for all $n$, by computing that the centralizer of $\F_n$ is trivial.

Moreover, Corollary~\ref{sut-relKlep} is applied to prove that the braid groups $B_\infty$ and $P_\infty$ on infinitely many strands are both $C^*$-simple.
For the latter, one checks the relative Kleppner condition for $(P_\infty,\F_\infty,1)$, and then for $B_\infty$ one checks the relative Kleppner condition for $(B_\infty,P_\infty,1)$.
\end{remark}

\section{Examples}\label{ex}

\subsection{Semidirect products of abelian groups by aperiodic automorphisms}\label{aperiodic-autos}

Throughout this subsection, $H$ will be an infinite abelian group and $\beta$ will denote an automorphism of $H$.
We will use addition to denote the group operation in $H$. Moreover, for $k\in \Z$ and $x\in H$, we will often write $k\cdot x$ instead of $\beta^k(x)$.
The automorphism $\beta$ will be called \emph{aperiodic} when the orbit of any nontrivial element in $H$ is infinite
(or, equivalently, when $k\cdot x \neq x$ for all $k\in\Z\setminus\{0\}$ and all $x\in H\setminus\{0\}$).

We will consider the semidirect product $G=H\rtimes \Z$ associated with the action of $\Z$ on $H$ induced by $\beta$. For further use, we note that for $x, y\in H$ and $k\in \Z$, we have
\begin{equation} \label{conj-eq}
(y,k)(x,0)(y,k)^{-1}= (y,0)(k\cdot x, 0)(y,0)^{-1}=(k\cdot x,0)\,.
\end{equation}
As usual, we will sometimes identify $H$ and $\Z$ with their canonical copies in $G$ via the maps $x \mapsto (x,0)$ and $k\mapsto (0,k)$, so that we may write \eqref{conj-eq} as
\[
(y,k) \,x \,(y,k)^{-1} = k\cdot x\,.
\]
In particular, we then have $kxk^{-1} = k\cdot x$ for $x\in H$ and $k\in \Z$, in agreement with the notation used in subsection~\ref{subgroups}.

Next, we remark that the following conditions are equivalent:
\begin{itemize}
\item[(i)] $\beta$ is aperiodic
\item[(ii)] $G$ is ICC
\end{itemize}
Indeed, if $\beta$ is not aperiodic, so there exists $x\in H\setminus \{0\}$ with a finite orbit in $H$, one easily sees from equation \eqref{conj-eq} that the conjugacy class of $x =(x,0)$ in $G$ is finite.
On the other hand, assume that $\beta$ is aperiodic.
If $x\in H\setminus\{0\}$, then
\[
\{ (0,l)(x,k)(0,l)^{-1} : l\in\Z \} = \{ (l\cdot x,k) : l\in\Z \}
\]
is clearly infinite for each $k\in \Z$. Further, if $k\in \Z \setminus \{0\}$, then
\begin{equation}\label{icc}
\{ (y,0)(0,k)(y,0)^{-1} : y\in H \} = \{ (y+k\cdot (-y),k) : y\in H \}
\end{equation}
is infinite. Indeed, if $y_1+k\cdot (-y_1)=y_2+k\cdot (-y_2)$, then $y_1-y_2=k\cdot (y_1-y_2)$, so $y_1=y_2$ as $\beta$ is aperiodic. Since $H$ is infinite, the claim holds. Thus we see that $G$ is ICC.

\medskip When $\beta$ is aperiodic, we thus get that the amenable group $G$, being ICC, does not lie in $\K$. However, as seen previously in Example~\ref{anosov} in the case where $G=\Z^n\rtimes_A \Z$, there can still exist $2$-cocycles $\sigma$ on $G$ such that $(G,\sigma)$ is $C^*$-simple and/or has the unique trace property. Our aim is to illustrate this in a more general context.

Let $\sigma'\in Z^2(H,\T)$. We will assume that $\sigma'$ is \emph{$\Z$-invariant}, meaning that it satisfies \[\sigma'(k\cdot x, k\cdot y) = \sigma'(x,y)\] for all $x,y\in H$ and $k\in\Z$. As is well-known, see e.g.\ \cite[Appendix~2]{PR} or \cite[2.1--2.4]{Om2} (and \cite{Om2b}), we may then define a $2$-cocycle $\sigma \in Z^2(G,\T)$ by
\[
\sigma\big((x, k), (y, l)\big) = \sigma'(x, k\cdot y)
\]
for $x, y \in H$ and $k,l \in \Z$. 
We then have that $\widetilde\sigma\big((0,k),(h,0)\big) =1$ for all $k \in \Z$ and $h \in H$, so it follows that $C_r^*(G,\sigma)$ decomposes as the reduced crossed product of $A=C_r^*(H,\sigma')$ by the action of $\Z$ on $A$ associated to the $*$-automorphism $\widetilde\beta$ of $A$ determined by $\widetilde\beta\big(\lambda_{\sigma'}(x)\big)= \lambda_{\sigma'}(\beta(x))$ for all $x \in H$.
We note that saying that $\Z=G/H$ acts on $A$ in a minimal way just means that $\widetilde\beta$ acts minimally on $A$, i.e., that the zero ideal is the only proper ideal of $A$ which is invariant under $\widetilde\beta$. 

To ease our analysis, we set
\[
S:=\{ x\in H : \text{$x$ is $\sigma'$-regular} \}.
\]
Since $H$ is abelian, we have $S=\{x \in H : \text{$\sigma'(x,y)=\sigma'(y,x)$ for all $y\in H$}\}$.
Moreover, $S$ is a subgroup of $H$ such that $k\cdot x \in S$ whenever $k\in \Z$ and $x\in S$ (since $\sigma'$ is invariant).

We also set $\sigma'':= (\sigma')_{|S\times S} \,= \sigma_{|S\times S} \in Z^2(S,\T)$.
As $\sigma''$ is a symmetric, it follows from \cite{Klepp} that $\sigma''$ is a coboundary, i.e., $\sigma'' \in B^2(S,\T)$, so $C^*_r(S,\sigma'')\simeq C_r^*(S)$ is commutative.

\begin{theorem}\label{aperiodic}
Let $H$, $\beta$, $G$, $\sigma$, and $\sigma'$ be as above and suppose that $\beta$ is aperiodic. Consider the following conditions:
\begin{itemize}\itemsep1pt
\item[(i)] $(H,\sigma')$ satisfies Kleppner's condition.
\item[(ii)] $(G,\sigma)$ has the unique trace property.
\item[(iii)] $(G,\sigma)$ is $C^*$-simple.
\end{itemize}
Then we have \textup{(i)}~$\Longleftrightarrow$~\textup{(ii)}~$\Longrightarrow$~\textup{(iii)}.
Moreover, if $H$ is countable, then \textup{(iii)} holds if and only if 
$\widetilde\beta$ acts minimally on $C_r^*(H,\sigma')$.
\end{theorem}

\begin{proof}
Suppose that $(x,k)\in G\setminus H$, i.e., $x\in H$ and $k\in\Z\setminus\{0\}$.
Then
\[
\{(y,0)(x,k)(y,0)^{-1} : y\in H\}=\{ (y+x+k\cdot(-y),k) : y\in H\}
\]
is infinite, since $\{ y+k\cdot(-y) : y\in H\}$ is infinite for every $k\in\Z\setminus\{0\}$ by a similar argument as the one used after \eqref{icc}. Thus it follows that $(G,H,1)$ satisfies the relative Kleppner condition.
Remark~\ref{relK}~(b) then implies that $(G,H,\sigma)$ always satisfies the relative Kleppner condition.
Hence, using Corollary~\ref{sut-relKlep-2} we get that (i)~$\Rightarrow$~(ii) (and also (i)~$\Rightarrow$~(iii)).
Since $G=H\rtimes \Z$ is amenable, Theorem~\ref{Murph} gives that (ii)~$\Rightarrow$~(iii).

To show the implication (ii)~$\Rightarrow$~(i), we first observe that $S$ is a normal subgroup of $G$. Hence, as in subsection~\ref{subgroups}, we get that for each $(y,n) \in G$, there exists a $*$-automorphism $\gamma_{(y,n)}$ of $C_r^*(S,\sigma'')$ satisfying
\begin{align*}
\gamma_{(y,n)}\big(\lambda_{\sigma''}(x)\big) &= \sigma((y,n),(x,0))\, \overline{\sigma\big((n\cdot x,0), (y,n)\big)}\,
\lambda_{\sigma''}(n\cdot x)\\
&= \sigma'(y,n\cdot x)\overline{\sigma'(n\cdot x, y)}\,\lambda_{\sigma''}(n\cdot x) \\
&= \lambda_{\sigma''}(n\cdot x)
\end{align*}
for all $x \in S$.

Set $\gamma=\gamma_{(0,1)}$. We then have $\gamma^n\big(\lambda_{\sigma''}(x)\big) = \lambda_{\sigma''}(n\cdot x)$ for all $n\in \Z$ and $x\in S$. Thus $n\mapsto \gamma^n$ is the $\Z$-action on $C_r^*(S,\sigma'')$ associated to the $\Z$-action on $S$ induced by the automorphism $\beta_S$ of $S$ given by $\beta_S(x) = \beta(x) = 1\cdot x$ for $x \in S$.

Assume now that (i) does not hold. Since $H$ is abelian, this means that $S$ is non-trivial. Since $\beta$ is aperiodic, $\beta_S$ is also aperiodic. Now, since $\sigma''$ is symmetric, we have
\[
\sum_{n=1}^\infty \Bigl\lvert 1-\sigma''(n\cdot x,y)\overline{\sigma''(y,n\cdot x)}\Bigr\rvert=0
\]
for all $x,y\in S$.
Therefore, combining \cite[Corollary~11.3.4]{NS} with \cite[Theorem~11.4.2]{NS} we get that there exists a $\gamma$-invariant state $\varphi$ on $C^*_r(S,\sigma'')$ different from the canonical tracial state $\tau''$.
Since $C^*_r(S,\sigma'')$ is commutative, $\varphi$ is automatically tracial.
Moreover, we have
\[
\varphi\Big(\gamma_{(y,n)}\big(\lambda_{\sigma''}(x)\big)\Big)= \varphi\big(\lambda_{\sigma''}(n\cdot x)\big)=\varphi\big(\gamma^n(\lambda_{\sigma''}(x))\big)=\varphi(\lambda_{\sigma''}(x))
\]
for all $(y,n)\in G$ and all $x\in S$. It follows then by linearity and continuity that $\varphi$ is invariant under each $\gamma_{(y,n)}$. If we now use subsection~\ref{subgroups} to decompose $C_r^*(G,\sigma)$ as
\[
C^*_r(G,\sigma)\simeq C_r^*\Big(C^*_r(S,\sigma''), G/S, \delta, \omega\Big)\, ,
\]
we can then conclude that $\varphi$ is $G/S$-invariant. Hence, letting $E_S$ denote the canonical conditional expectation from $C_r^*(G,\sigma)$ onto $C_r^*(S,\sigma'')$, we obtain that $\tilde\varphi:=\varphi\circ E_S$ is a tracial state on $C_r^*(G, \sigma)$, which is different from the canonical one since the restriction of $\tilde\varphi$ to $C_r^*(S, \sigma'')$ is different from $\tau''$. Thus (ii) does not hold.

To show the final assertion, assume that $H$ is countable. As $(G, H, \sigma)$ satisfies the relative Kleppner condition, Corollary~\ref{sut-torsionfree} gives that $(G,\sigma)$ is $C^*$-simple whenever $\Z=G/H$ acts on $C^*_r(H,\sigma')$ in a minimal way, i.e., whenever $\widetilde\beta$ acts minimally on $C^*_r(H,\sigma')$. The converse statement also holds, as may be seen by writing $C_r^*(G, \sigma)$ as a reduced crossed product over $C_r^*(H, \sigma')$.
\end{proof}

\begin{remark}\label{invariant-b}
In the situation of Theorem~\ref{aperiodic}, we do not know whether (iii)~$\Rightarrow$~(i), or, equivalently, whether (iii)~$\Rightarrow$~(ii). The following discussion sheds some light on this problem. Suppose that (i) does not hold, so $S$ is nontrivial, and in fact infinite. As $\sigma''$ is a coboundary, there exists a function $b\colon S\to \T$ such that $b(0)=1$ and $\sigma''(x,y) = b(x)b(y)\overline{b(x+y)}$ for all $x,y\in S$. Assume that we can choose $b$ in such a way that there exists some $m\in\Z\setminus\{0\}$ such that $b(-m\cdot x)=b(x)$ for all $x\in S$. Then $(G,\sigma)$ is not $C^*$-simple.

To verify this, we first extend $b$ to $c\colon G\to\T$ by setting
\[
c(x,n)=\begin{cases}b(x)&\text{for $x\in S$ and $n\in\Z$,}\\1&\text{for $x\in N\setminus S$ and $n\in\Z$.}\end{cases}
\]
To lighten our notation, we will just write $yn$ for an element $(y,n)\in G$ from now on.
Let then $\rho \in B^2(G,\T)$ be the coboundary associated to $c$ and set $\omega:=\sigma\overline\rho\sim\sigma$.
Note that $\omega(x,y)=1$ for all $x,y\in S$.
According to \cite[Theorem~1.5]{PR}, there is an action of $G$ on $\widehat S$ (the Pontryagin dual of $S$) given by
\begin{equation} \label{pr-action}
(yn \cdot \psi)(x) = \omega(x,(yn))\,\overline{\omega((yn),(yn)^{-1}x(yn))}\,\psi((yn)^{-1}x(yn)),
\end{equation}
for $y\in N$, $n\in \Z$ (i.e., $yn\in G$), $\psi\in\widehat S$ and $x\in S$.
Letting $1$ denote the trivial character on $S$, we then get
\[
\begin{split}
(n \cdot 1)(x)
&= \omega(x,n)\overline{\omega(n,(-n)\cdot x)}\\
&= \sigma'(x,0)\overline{c(x)c(n)}c(xn)\overline{\sigma'(0,x)}c(n)c(-n\cdot x)\overline{c(xn)}\\
&= b(-n\cdot x)\overline{b(x)}\,
\end{split}
\]
for all $n\in \Z$ and $x\in S$. Using our assumption on $b$, we thus get that $m \cdot 1= 1$.
Hence, the orbit of $1$ in $\widehat S$ under the action of $\Z$ is finite. Since $\widehat S$ is infinite, this implies that $\Z$ does not act minimally on $\widehat S$.
Hence, \cite[Theorem~1.5]{PR} gives that $(G,\omega)$ is not $C^*$-simple, and it follows that $(G,\sigma)$ is not $C^*$-simple.
Equivalently, this shows that $\widetilde\beta$ does not act minimally on $C_r^*(H,\sigma')$.

It is unclear to us whether it is always possible to choose $b$ as above.
\end{remark}

\begin{example}\label{Zn-by-Z}
Consider the case where $H=\Z^n$ and $\beta(x) = Ax$ for a matrix $A\in GL(n,\Z)$ such that $\beta$ is aperiodic. One can then deduce from \cite[Proposition~3.1]{PR} that, up to similarity, any $\sigma \in Z^2(\Z^n\rtimes_A \Z,\T)$ arises from some $\Z$-invariant $\sigma' \in Z^2(\Z^n, \T)$. Moreover, all three conditions in Theorem~\ref{aperiodic} are then equivalent. Indeed, assume (i) does not hold, i.e., $S\neq \{0\}$, and let $\omega\sim\sigma$ be such that $\omega_{|S\times S}=1$.
As $\beta$ is aperiodic, $A-I$ is not nilpotent, so \cite[Remark~3.3]{PR} gives that the action of $\Z$ on $\widehat S$ (defined as in equation \eqref{pr-action}) is not minimal. It follows then from \cite[Theorem~3.2]{PR} that $C^*_r(\Z^n\rtimes_A \Z,\omega)$ is not simple, and hence $C^*_r(\Z^n\rtimes_A \Z,\sigma)$ is not simple.
\end{example}

\subsection{Wreath products}\label{wreath}

Let $N$ and $K$ be nontrivial groups.
We recall that the \emph{wreath product} $N\wr K$ is defined as the semidirect product $\left(\bigoplus_K N\right) \rtimes K$,
where $K$ acts by (left) translation on the index set, that is, by
\[
\left(k\cdot (x_j)_{j\in K} \right)_l=x_{k^{-1}l}\,, \text{ or, equivalently, by }\, k\cdot (x_j)_{j\in K}=(x_{k^{-1}j})_{j\in K}.
\]
We start by recording a useful result.

\begin{lemma}\label{wreath relK}
The triple $(N\wr K,\bigoplus_K N,1)$ satisfies the relative Kleppner condition if and only if $K$ or $N$ is infinite.
\end{lemma}

\begin{proof}
If $y\in (N\wr K)\setminus\bigoplus_K N$, that is, $y=\left((y_j)_{j\in K},k\right)$, where $k\neq e$,
and $x=\left((x_j)_{j\in K},e\right)\in\bigoplus_K N$, then
\[
xyx^{-1}
=\left((x_j)_{j\in K},e\right)\left((y_j)_{j\in K},k\right)\big((x_j^{-1})_{j\in K},e\big)
=\big((x_jy_jx_{k^{-1}j}^{-1})_{j\in K},k\big).
\]
If $\bigoplus_K N$ is infinite, by letting $(x_j)_{j\in K}$ vary, this takes an infinite number of values.
To see this, note first that $\bigoplus_K N$ is infinite whenever $N$ or $K$ is infinite.
If $N$ is infinite, then it suffices to fix one $l\in K$ and consider all sequences $(x_j)_{j\in K}$ with $x_j=e$ if $j\neq l$.
On the other hand, if $N$ is finite, then $K$ is infinite, so we fix a nontrivial $h\in N$,
and consider all sequences $(x_j)_{j\in K}$ such that for some finite $F\subset N$, $x_j=h$ for $j\in F$ and $x_j=e$.
\end{proof}

With a similar argument, one can show that $N\wr K$ is ICC if and only if $K$ is infinite or $N$ is ICC
(cf.~\cite[Corollary~4.2]{Preaux}).

\begin{proposition}
The wreath product $N\wr K$ is $C^*$-simple (resp.\ has the unique trace property) if and only if $N$ is $C^*$-simple (resp.\ has the unique trace property).
\end{proposition}

\begin{proof}
If $N\wr K$ is $C^*$-simple, then the normal subgroup $\bigoplus_K N$ is $C^*$-simple \cite[Theorem~3.14]{BKKO},
and (the canonical copy of) $N$ is normal in $\bigoplus_K N$, so it is $C^*$-simple as well.

If $N$ is $C^*$-simple, then the direct sum $\bigoplus_K N$ is $C^*$-simple \cite[Corollary~II.8.2.5]{Bla} and $N$ is infinite, so it follows from Lemma~\ref{wreath relK} and Corollary~\ref{sut-relKlep} that $N\wr K$ is $C^*$-simple.

A similar argument works for the unique trace property.
\end{proof}

A description of $H^2(N\wr K,\T)$ may be deduced from a result of Tappe, \cite[Corollary on~p.~2]{Tappe}, where he deals with a more general situation: he lets $K$ acts on an index set $I$, while we only consider the case where $I=K$ and $K$ acts on itelf by (left) translation.

Let $H^2\left(\bigoplus\nolimits_K N,\T\right)^K$ denote the elements in $H^2(\bigoplus\nolimits_K N, \T)$ that are invariant under the natural action of $K$ induced from its action on $\bigoplus\nolimits_K N$. Then Tappe's result says first that
\[
H^2(N\wr K,\T)\ \simeq\ H^2(K,\T) \times H^2\left(\bigoplus\nolimits_K N,\T\right)^K.
\]
Moreover, when $K$ has no nontrivial elements of order two, as will be the case in the examples we consider, the summand $H^2\left(\bigoplus\nolimits_K N,\T\right)^K$ may be described as follows.
Let $B(N,N)$ denote the group of bihomomorphisms from $N \times N$ into $\T$ (which is isomorphic to the dual group of $H_1(N) \otimes_\Z H_1(N)$).
Further, let $I_2$ denote the family of all subsets of $K$ containing two distinct elements.
Then $K$ acts on $I_2$ by translation, and we let $I_2/K$ denote the associated orbit space.
We then have
\[
H^2\left(\bigoplus\nolimits_K N,\T\right)^K \simeq\ H^2(N,\T) \times \prod_{I_2/K} B(N,N)\,.
\]
When $K$ has nontrivial elements of order two, an extra summand $C$ appears, and we refer to Tappe's article for further details.
Summarizing this discussion, we have:

\begin{lemma}\label{wreath cohomology}
Assume $K$ has no nontrivial element of order two. Then
\[
H^2(N\wr K,\T)\ \simeq\ H^2(K,\T) \times H^2(N,\T) \times \prod_{I_2/K} B(N,N)\,.
\]
\end{lemma}

We now consider the situation where $N$ is abelian and $K=\Z$. Then $H:=\bigoplus_\Z N$ is abelian and the action of $K=\Z$ on $H$ clearly arises from an aperiodic automorphism of $H$. Hence the wreath product $N\wr\Z = H\rtimes \Z$ fits within the set-up of the previous subsection. 
If $\omega$ is a $2$-cocycle on $H= \bigoplus_\Z N$ which is invariant under the action of $\Z$, then $\check\omega$ will denote the induced $2$-cocycle on $N\wr\Z$ given by
\[
\check\omega\Big(\big((x_j)_{j\in \Z}\,,m\big), \big((y_j)_{j\in \Z}\,,n\big)\Big)=\omega\big((x_j)_{j\in \Z}\,,m\cdot (y_j)_{j\in \Z}\big).
\]
Since $H^2(\Z,\T)=\{1\}$, every $2$-cocycle on $N\wr\Z$ is similar to one that arises this way.

\begin{proposition}\label{wreath-ab-Z}
Assume that $N$ is abelian and let $\sigma$ be a $2$-cocycle on $N\wr\Z$. Let $\sigma'$ denote its restriction to $H=\bigoplus_\Z N$.
Consider the following conditions:
\begin{itemize}\itemsep1pt
\item[(i)] $(H,\sigma')$ satisfies Kleppner's condition.
\item[(ii)] $(N\wr\Z,\sigma)$ has the unique trace property.
\item[(iii)] $(N\wr\Z,\sigma)$ is $C^*$-simple.
\end{itemize}
Then we have \textup{(i)}~$\Longleftrightarrow$~\textup{(ii)}~$\Longrightarrow$~\textup{(iii)}.
Moreover, if $N$ is countable, then \textup{(iii)} holds if and only if the associated action of $\Z$ on $C_r^*(H,\sigma')$ is minimal.
\end{proposition}

\begin{proof}
By Tappe's result mentioned above, there exists an invariant $\omega\in Z^2(H,\T)$ such that $\sigma$ is similar to $\check\omega$ via some coboundary $\rho \in B^2(N\wr \Z, \T)$. Then $\sigma'$ is similar to $\omega$ via $\rho_{|H\times H}$, and the result follows from Theorem~\ref{aperiodic}.
\end{proof}

In concrete cases, it is possible to be more specific. We illustrate this by choosing first $N=\Z$, then $N=\Z_2$.

\subsubsection{The group $\Z\wr \Z$}

First, before we discuss $C^*$-simplicity and the unique trace property of $\Z\wr \Z$, we compute its $2$-cocycles up to similarity, by using results of the previous subsections.
Since the wreath product $\Z\wr\Z$ is given as $\left(\bigoplus\nolimits_\Z \Z\right) \rtimes \Z$, we first look at the group $\bigoplus_\Z \Z=\bigoplus_{-\infty}^\infty \Z$ and its second cohomology group.
The elements of $\bigoplus_\Z \Z$ are sequences $x=(x_j)_{j=-\infty}^\infty$, where $x_j\in\Z$ for all $j\in\Z$, and $x_j\neq 0$ only for finitely many $j$'s.
For each $k\in\Z$, we will let $e_k$ denote the sequence in $\bigoplus_\Z \Z$ where $(e_k)_j=\delta_{jk}$.
Gelfand theory gives that the group C$^*$-algebra of $\bigoplus_\Z \Z$ is isomorphic to $C(\T^{\Z})$, where $\T^{\Z}$ denotes the infinite-dimensional torus $\prod_{j\in \Z} \T$.
When $\sigma' \in Z^2(\bigoplus_\Z \Z, \T)$ is not similar to $1$, we may therefore think of $C_r^*(\bigoplus_\Z \Z, \sigma')$ as a noncommutative infinite-dimensional torus.

Standard properties of group cohomology give that
\begin{equation}\label{torus cohomology}
\begin{split}
H^2\Big(\bigoplus_{-\infty}^\infty\Z,\T\Big)
&=H^2\Big(\varinjlim \bigoplus_{-n}^n \Z,\T\Big)
=\varprojlim H^2\Big(\bigoplus_{-n}^n \Z,\T\Big)\\
&=\varprojlim \T^{\frac{1}{2}n(n-1)}
=\prod_I\T,
\end{split}
\end{equation}
where the index set $I$ is $\{(j,k)\in\Z^2 \mid j<k\}$.
It follows that every element of $Z^2\Big(\bigoplus_\Z \Z,\T\Big)$ is similar to one of the form
\begin{equation}\label{torus matrix}
\sigma_\theta\Big((x_j)_{j\in\Z},(y_j)_{j\in\Z}\Big)=\prod_{j<k}e^{2\pi i\,\theta_{j,k}\,x_jy_k},
\end{equation}
where $\theta=(\theta_{j,k})$ is an upper triangular $\Z\times\Z$-matrix with $\theta_{j,k}\in[0,1)$ whenever $j<k$.

As $\bigoplus_\Z \Z$ is abelian,
$(\bigoplus_\Z \Z,\sigma_\theta)$ is $C^*$-simple (resp.\ has the unique trace property) if and only if Kleppner's condition holds for $(\bigoplus_\Z \Z,\sigma_\theta)$.
It is not easy to express this condition in terms of $\theta$ (this is already the case when considering $\bigoplus_{j=1}^n \Z = \Z^n$ for finite $n\geq 4$).
However, we remark that if Kleppner's condition holds for $(\bigoplus_\Z \Z,\sigma_\theta)$, then for all $k\geq 1$,
the subgroup $S_k$ generated by $\{e^{2\pi i\theta_{j,k}},e^{2\pi i\theta_{k,j}}:j\in\Z\}$ must be dense in $\T$.
Indeed, if this is not the case, there exist $k,m\geq 1$ such that $(S_k)^m=\{1\}$, and then $m\,e_k$ is $\sigma$-regular.
Moreover, as opposed to the situation for finite direct sums of $\Z$,
Kleppner's condition may hold even when all entries $\theta_{jk}$ of $\theta$ are rational, cf.~Example~\ref{exmp ZwrZ}~(d).

\medskip

For a given $\theta$ as above, consider the homomorphism
\[
T_\theta\colon \bigoplus_{-\infty}^\infty\Z\longrightarrow\prod_{-\infty}^\infty\T
\]
defined as the composition
\[
\bigoplus_{-\infty}^\infty\Z
\longrightarrow\bigoplus_{-\infty}^\infty\R
\longrightarrow\prod_{-\infty}^\infty\R
\longrightarrow\prod_{-\infty}^\infty\T,
\]
where the first map is the inclusion map, 
the middle one is the map $x\mapsto (\theta-\theta^*)x$,
where $\theta^*$ denotes the transpose of $\theta$, and the third is the quotient map, mapping $(r_k)_{k\in \Z} \in \prod_\Z\R$ to $\big(e^{2\pi i \, r_k}\big)_{k\in \Z} \in \prod_\Z\T$.
Then $(\bigoplus_\Z \Z,\sigma_\theta)$ satisfies Kleppner's condition if and only if $T_\theta$ is injective.
Indeed, $x$ is $\sigma_\theta$-regular if and only if $\sigma(x,e_k)=\sigma(e_k,x)$ for all $k\in\Z$, i.e.,
if and only if
\[
1= \overline{\sigma(x,e_k)}\sigma(e_k,x)=\prod_{j<k}e^{2\pi ix_j\theta_{j,k}}\prod_{k<l}e^{-2\pi ix_l\theta_{k,l}}=e^{2\pi i\, e_k^*(\theta-\theta^*)x}
\]
for all $k\in \Z$.
That is, the kernel of $T_\theta$ consists precisely of all the $\sigma_\theta$-regular elements.

\medskip

Next, we consider
\[
\Z\wr\Z=\left(\bigoplus\nolimits_\Z \Z\right) \rtimes \Z,
\]
where we recall that $\Z$ acts on $\bigoplus\nolimits_{\Z} \Z$ by
\[
\left(n\cdot (x_j)_{j\in\Z}\right)_k=x_{k-n}.
\]
In particular, $n\cdot e_k = e_{k+n}$ for $k,n \in \Z$.
The $2$-cocycle $\sigma_\theta$ on $\bigoplus_\Z \Z$ is invariant under the induced action of $\Z$ if and only if for all integers $j<k$ and $n$ we have
\[
e^{2\pi i\,\theta_{j,k}}=\sigma(e_j,e_k)
=\sigma(n\cdot e_j,\,n\cdot e_k)
=\sigma(e_{j+n},e_{k+n})=e^{2\pi i\, \theta_{j+n,k+n}}.
\]
That is, $\sigma_{\theta}$ is invariant if and only if $\theta_{jk}=\theta_{j+n,k+n}$ for all integers $j<k$ and $n$, i.e., if and only if the matrix $\theta$ is constant on its diagonals.
Setting $\theta_m=\theta_{0, m}$ for each integer $m \geq 1$, this means that we have $\theta_{j,k}= \theta_{k-j}$ when $j<k$ and is $0$ otherwise.
It follows from Lemma~\ref{wreath cohomology} that
\[
H^2(\Z\wr\Z,\T)\simeq H^2\left(\bigoplus\nolimits_\Z \Z,\T\right)^\Z\simeq\prod_{m=1}^{\infty}\T.
\]
Hence, any element of $Z^2(\Z\wr\Z,\T)$ is, up to similarity, of the form $\check\sigma_\theta$, where
\begin{equation}\label{check-sigma}
\check\sigma_\theta\Big(\big((x_j)_{j\in\Z}, \,n\big), \big((y_j)_{j\in\Z}, \,n'\big)\Big) = \sigma_\theta\Big( (x_j)_{j\in\Z}\,, \,n\cdot (y_j)_{j\in\Z}\Big)
\end{equation}
and $\theta$ is an upper triangular $\Z\times\Z$-matrix which is constant on its diagonals, i.e., such that $\theta_{j,k}= \theta_{k-j}$ when $j<k$ for some sequence $\{\theta_m\}_{m\in \N}$ in $[0,1)$.

Applying Proposition~\ref{wreath-ab-Z} we get:
\begin{proposition}\label{propZwrZ}
Assume that $\theta$ is constant on its diagonals and $\check\sigma_\theta$ is as in \eqref{check-sigma}.
Then $(\Z\wr\Z,\check\sigma_\theta)$ has the unique trace property if and only if $(\bigoplus_\Z \Z,\sigma_\theta)$ satisfies Kleppner's condition, which implies that $(\Z\wr\Z,\check\sigma_\theta)$ is $C^*$-simple.
\end{proposition}

\begin{example}\label{exmp ZwrZ}
Here we provide some insight on Kleppner's condition for $(\bigoplus_\Z \Z,\sigma_\theta)$ when the matrix $\theta$ is of the form described just before Proposition~\ref{propZwrZ}.
\begin{itemize}
\item[(a)]
First, we note that for every $k\geq 1$,
the group $S_k$ (as defined previously) coincide with the subgroup $S$ of $\T$ generated by $\{e^{2\pi i\, \theta_m}:m\in \N\}$.
Thus, density of $S$ in $\T$ is necessary (but not sufficient) for Kleppner's condition to hold for $(\bigoplus_\Z \Z,\sigma_\theta)$.
\item[(b)]
If $\theta_m\neq 0$ only for finitely many indices, then density of $S$ is also sufficient.
Clearly, in this case $S$ is dense in $\T$ if and only if it $\theta_m$ is irrational for some $m \in \N$.
Let us assume this holds, and let $n$ be the largest number for which $\theta_n$ is irrational.
Suppose that $x$ is $\sigma$-regular and assume (for contradiction) that $x$ has some nonzero terms.
Let $k$ be the largest index with $x_k\neq 0$.
Then
\[
1=\overline{\sigma_\theta(e_{n+k},x)}\sigma_\theta(x,e_{n+k})=1\cdot\prod_{j\leq k}e^{2\pi i\,\theta_{n+k-j}x_j},
\]
and only $\theta_n,\theta_{n+1},\dotsc$ appear in the expression above, so if $\theta_n$ is the only irrational number among these,
it follows that $x_k=0$, which gives a contradiction.
\item[(c)]
To see why density of $S$ in $\T$ in general is not sufficient,
take $r$ to be an irrational number in $(0,1)$,
and for $k\geq 0$ set
\[
\theta_{4k+1}=r, \quad \theta_{4k+3}=1-r, \quad \text{and} \quad \theta_{2k}=0.
\]
Then $e_1+e_3$ is $\sigma_\theta$-regular.
In fact, $e_1^*(\theta-\theta^*)=-e_3^*(\theta-\theta^*)$, i.e.,
column $1$ and $3$ of the matrix $\theta-\theta^*$ are the negative of each other.
\item[(d)]
Let $p_1<p_2<p_3<\dotsb$ denote the list of all prime numbers and define $\theta_m=\frac{1}{p_m}$ for every $m\geq 1$.

Then $(\bigoplus_\Z \Z,\sigma_\theta)$ satisfies Kleppner's condition.
Indeed, suppose that $x$ is $\sigma_\theta$-regular and choose $n$ so large that $p_n>\sum_{j\in\Z}\,\lvert x_j\rvert$.
Assume, for contradiction, that $x$ is nontrivial,
and let $k'$ and $k$ denote respectively the smallest and the largest number in the set $\{j\in \Z : x_j\neq 0\}$.
Then 
\[
1=\overline{\sigma(e_{n+k},x)}\sigma(x,e_{n+k})= e^{2\pi i\,\Big(\sum_{j=k'}^{k}\frac{x_j}{p_{n+k-j}}\Big)}
\]
and $\Big\lvert\sum_{j=k'}^{k}\frac{x_j}{p_{n+k-j}}\Big\rvert<1$ by assumption, so the sum must be $0$.
But this is not possible unless all $x_j$'s in this sum are $0$.
Indeed, one easily checks that $\frac{x_{k}}{p_n}\notin\Z\big[\{\frac{1}{p_j}:j>n\}\big]$ when $0<\lvert x_{k}\rvert<p_n$, so that we must have $x_k=0$.
Proceeding inductively, we also get $x_{k-1}=\cdots=x_{k'}=0$. Thus, $x$ must be trivial, giving a contradiction.
\end{itemize}
\end{example}

\begin{remark}
Since $G=\Z\wr\Z$ is ICC and amenable, we have $C^*S(G) \neq K(G)= Z^2(G,\T)$. Moreover, 
Proposition~\ref{propZwrZ} and Example~\ref{exmp ZwrZ} give that $C^*S(G)\neq \emptyset$.
Similarly, we have $\emptyset \neq UT(G)\neq K(G)$.
\end{remark}

\begin{remark}
When the matrix $\theta$ in Proposition~\ref{propZwrZ} is such that $(\bigoplus_\Z \Z,\sigma_\theta)$ does not satisfy Kleppner's condition, we do not know if it can happen that $\Z$ acts on $C_r^*(\bigoplus_\Z \Z,\sigma_\theta)$ in a minimal way; this would imply that $(\Z\wr\Z,\check\sigma_\theta)$ is $C^*$-simple (cf.~Proposition~\ref{wreath-ab-Z}) without having the unique trace property.
\end{remark}

\subsubsection{The lamplighter group $\Z_2\wr\Z$}

Analogously to the previous example, we start by computing the $2$-cocycles of $\Z_2\wr\Z$ up to similarity.
Let $\bigoplus_\Z \Z_2$ denote the direct sum of $\Z_2$ indexed by $\Z$.
As in \eqref{torus cohomology}, we get that $H^2(\bigoplus_\Z \Z_2,\T)\simeq\prod_I\Z_2^\times$,
where the index set $I$ is $\{(j,k)\in\Z^2 \mid j<k\}$.
We will represent its elements by $\Z\times\Z\,$-matrices of the form $\mu=\big[\mu_{jk}\big]_{j,k\in\Z}$\,,
where $\mu_{jk}=1$ whenever $j\geq k$ and $\mu_{jk}\in\{-1,1\}$ if $j<k$.
Analogously to \eqref{torus matrix}, every element of $Z^2\big(\bigoplus_\Z \Z_2\,,\T\big)$ is similar to one of the form
\[
\sigma_\mu\big((s_j)_{j\in\Z}\,,(t_j)_{j\in\Z}\big)=\prod_{j<k}\mu_{jk}^{s_jt_k}.
\]
Consider now the lamplighter group
\[
\Z_2\wr\Z=\left(\bigoplus\nolimits_\Z \Z_2\right) \rtimes \Z,
\]
where the action of $\Z$ on $\bigoplus\nolimits_\Z \Z_2$ is given by
\[
\big(n\cdot (s_j)_{j\in\Z}\big)_k=s_{k-n}\,
\]
for $k, n \in \Z$. The following mirrors the previous subsection.
The $2$-cocycle $\sigma_\mu$ of $\bigoplus_\Z \Z_2$ is invariant under the action of $\Z$ 
if and only if $\mu_{jk}=\mu_{j+n,k+n}$ for all $j<k$ and $n\in\Z$,
i.e., if the matrix $\mu$ is constant on its diagonals.
Moreover, up to similarity, every $2$-cocycle of $\Z_2\wr\Z$ is similar to a $2$-cocycle $\check\sigma_\mu$ given by
\[\check\sigma_\mu \Big(\big((s_j)_{j\in\Z}, \,n\big), \big((t_j)_{j\in\Z}, \,n'\big)\Big) = \sigma_\mu\Big( (s_j)_{j\in\Z}\,, \,n\cdot (t_j)_{j\in\Z}\Big)\,\]
for some $\mu$ which is constant on its diagonals.
In other words, we have
\[
H^2(\Z_2\wr\Z,\T)\,\simeq\, H^2\left(\bigoplus\nolimits_\Z \Z_2,\T\right)^\Z\, \simeq \, \prod_\N\Z_2^\times.
\]
We assume from now on that $\mu$ is constant on its diagonals.
C$^*$-algebras of the form $C_r^*(\bigoplus_\Z \Z_2,\sigma_\mu)$ for such $\mu$'s have been previously discussed in the literature as ``C$^*$-algebras of bitstreams'', see for example \cite[Section~12]{NS}.
Letting $\mu_n$ denote the entry of $\mu$ on its $n$'th diagonal for each integer $n\geq 1$, the associated ``bitstream'' $\{\epsilon_n\}_{n=1}^\infty \in \{0,1\}^\N$ is given by setting $\epsilon_n=0$ if $\mu_n=1$ and $\epsilon_n=1$ if $\mu_n=-1$.
Set
\[
X_\mu:=\{n\geq 1 : \epsilon_n=1\}=\{n\geq 1:\mu_n=-1\}
\]
and $Y_\mu:=X_\mu\cup (-X_\mu)=\{\pm \, n:n\in X_\mu\} \subset \Z$.
As in \cite{NS} we will say that $X_\mu$ is \emph{periodic} if $Y_\mu$ is periodic, i.e., if there exists an integer $m\geq 1$ such that $\{m+y:y\in Y_\mu\}=Y_\mu$.
It follows from \cite[Corollary~12.1.5]{NS} that $(\bigoplus_\Z \Z_2,\sigma_\mu)$ is $C^*$-simple (resp.\ has the unique trace property) if and only if $X_\mu$ is nonperiodic, in which case $C_r^*(\bigoplus_\Z \Z_2,\sigma_\mu)$ is the UHF algebra of type $2^\infty$.
Since $\bigoplus_\Z \Z_2$ is abelian, this means that $(\bigoplus_\Z \Z_2,\sigma_\mu)$ satisfies Kleppner's condition if and only if $X_\mu$ is nonperiodic.

Nonperiodic $X_\mu$'s are easy to produce. This happens for example when $X_\mu$ is finite and nonempty. Since $0\notin Y_\mu$, this is also happens when $\mu_n=-1$ for every even $n\geq 1$. On the other hand, if $\mu_n=-1$ for every odd $n\geq 1$ and $\mu_n=1$ otherwise, i.e., $X_\mu = \N\setminus 2\N$, then $X_\mu$ is periodic.

From Proposition~\ref{wreath-ab-Z} we now get:
\begin{proposition}\label{lamp-prop}
Assume that $\mu$ is constant on its diagonals. Then the ``noncommutative lamplighter'' $(\Z_2\wr\Z\,,\check\sigma_\mu)$ has the unique trace property if and only if $X_\mu$ is nonperiodic, which implies that $(\Z_2\wr\Z\,,\check\sigma_\mu)$ is $C^*$-simple.
\end{proposition}

\begin{remark}
Suppose that $X_\mu=\N\setminus 2\N$, and let $\sigma_\mu$ be the associated $2$-cocycle.
Then $X_\mu$ is periodic, as indicated above, and
\[
\sigma_\mu(e_i+e_{i+2},x)=(-1)^{x_{i+1}}=\sigma_\mu(x,e_i+e_{i+2}).
\]
for all $x\in H=\bigoplus_\Z\Z_2$.
We can now check that $S=\langle e_i+e_{i+2} : i\in\Z \rangle\subset H$.

Given an element $x\in H$ and $i\in\{0,1\}$, define $x^i$ by
\[
(x^i)_k=\begin{cases} x_k & \text{if $k\in 2\Z+i$,} \\ 0 & \text{else.} \end{cases}
\]
Then $x=x^0+x^1$ and using that $\sigma_\mu$ is a bicharacter, we have
\[
\sigma_\mu(x,y)=\sigma_\mu(x^0+x^1,y^0+y^1)=\sigma_\mu(x^0,y^1)\sigma_\mu(x^1,y^0),
\]
since $\sigma_\mu(x^0,y^0)=\sigma_\mu(x^1,y^1)=1$ for all $x,y\in H$.

Define $b\colon H\to\T$ by
\[
b(x)=\sigma_\mu(x^0,x^1).
\]
Let $x,y\in S$ and note that in this case we have
\[
\sigma_\mu(x^1,y^0)=\sigma_\mu(x,y^0)=\sigma_\mu(y^0,x)=\sigma_\mu(y^0,x^1).
\]
We compute that
\[
\begin{split}
b(x+y) &= \sigma_\mu(x^0+y^0,x^1+y^1) \\
&= \sigma_\mu(x^0,x^1)\sigma_\mu(x^0,y^1)\sigma_\mu(y^0,x^1)\sigma_\mu(y^0,y^1) \\
&= b(x)\sigma_\mu(x,y)b(y).
\end{split}
\]
Thus, $(\sigma_\mu)_{|S\times S}$ coincides with the coboundary associated with $b$.
As is easy to check, $b$ is invariant, i.e., $b(1\cdot x) = b(x)$ for all $x \in S$. So the argument of Remark~\ref{invariant-b} applies with $m=-1$, and it follows that $(\Z_2\wr\Z,\check\sigma_\mu)$ is not $C^*$-simple.
It is possible that one could argue along the same lines whenever $X_\mu$ is periodic, but this might be combinatorially much more involved,
and we leave this as an open problem. An alternative way to proceed could be to show that $\Z$ does not act on $C_r^*(\bigoplus_\Z \Z_2,\sigma_\mu)$ in a minimal way when $X_\mu$ is periodic.
\end{remark}

\subsection{The Sanov transformation group} \label{Z2xF2}

As is well known, the two matrices
\[
\begin{bmatrix}
1 & 2 \\ 0 & 1
\end{bmatrix}
\quad\text{and}\quad
\begin{bmatrix}
1 & 0 \\ 2 & 1
\end{bmatrix}
\]
generate a free subgroup of $\op{SL}(2,\Z)$, sometimes called the Sanov subgroup.
We just denote this group $\F_2$ and its generators $v_1$ and $v_2$,
and consider the semidirect product $G=\Z^2\rtimes\F_2$ obtained via the canonical action of $\op{SL}(2,\Z)$ on $\Z^2$.
It is easy to verify that $G$ is ICC.
We use $e_1=(1,0)$ and $e_2=(0,1)$ to denote the generators of $\Z^2$.

To compute the $2$-cocycles of $G$ up to similarity, one may use Mackey type results as described in \cite[2.1-2.4]{Om2} (see also \cite[Theorem~2.1 and Proposition~2.2]{Om2b}).
Note that up to similarity, every $2$-cocycle of $\F_2$ is trivial, and every $2$-cocycle of $\Z^2$ is (uniquely) similar to one of the form
\begin{equation}\label{restricted}
\sigma_0((a_1,a_2),(b_1,b_2))=\mu_0^{\frac{1}{2}(a_1b_2-a_2b_1)}
\end{equation}
for some $\mu_0\in\T$.

One gets that every $2$-cocycle on $G$ is similar to one given by
\[
\sigma((a,x),(b,y))=\sigma_0(a,x\cdot b)g(b,x),
\]
where $\sigma_0$ is of the form \eqref{restricted}, and $g\colon \Z^2 \times \F_2 \to \T$ is a function satisfying
\begin{equation}\label{g-identities}
\begin{gathered}
g(a+b,x)=g(a,x)g(b,x), \\
g(a,xy) = g(y\cdot a,x)g(a,y),\\
g(0, x) = g(a, 1) = 1,\\
g(e_1,v_2)=g(e_2,v_1)=1.
\end{gathered}
\end{equation}
It follows that $g$ is uniquely determined by the two values $g(e_1,v_1)=\mu_1$ and $g(e_2,v_2)=\mu_2$ and
one deduces then without much trouble that $H^2(G,\T)\cong\T^3$.

We will therefore assume that $\sigma$ is of the form described above, hence is determined by $\mu_0, \mu_1, \mu_2 \in \T$,
and consider the decomposition
\[
C_r^*(G,\sigma) \simeq C_r^*\big(C_r(\Z^2, \sigma_0), \F_2, \beta, \omega\big)
\]
obtained as in subsection~\ref{subgroups}, using the section $s\colon\F_2 \to G$ given by $s(x) = (0, x)$. 
Straightforward computations give that $\omega$ is trivial and $\beta_x( \lambda_{\sigma_0}(a)) = g(a, x)\, \lambda_{\sigma_0}(x\cdot a)$ for all $a \in \Z^2$ and $x \in \F_2$.

Assume first that $\mu_0$ is nontorsion.
Then $(\Z^2,\sigma_0)$ is $C^*$-simple and has the unique trace property.
Since $\F_2$ is $C^*$-simple, it follows from Proposition~\ref{BryKed} that $(G,\sigma)$ is $C^*$-simple and has the unique trace property.

Next, we assume that $1\leq p<q$ are integers with $\gcd(p,q)=1$ and $\mu_0=e^{2\pi i p/q}$.
Then $C_r^*(\Z^2,\sigma_0)$ is a rational noncommutative $2$-torus with generators $U_1=\lambda_{\sigma_0}(e_1)$ and $U_2=\lambda_{\sigma_0}(e_2)$.
It is well known that the center $Z$ of $C_r^*(\Z^2,\sigma_0)$ is the C$^*$-subalgebra generated by $U_1^q$ and $U_2^q$, so $Z \simeq C(\T^2)$.
It is also known that $\op{Prim}(C_r^*(\Z^2,\sigma_0))$ is homeomorphic to $\T^2$ (see e.g.~\cite[Example~8.46]{Wil}).
Hence, using Remark~\ref{Prim}, we see that $\F_2$ will act on $C^*_r(\Z^2,\sigma_0)$ in a minimal way
whenever there is no proper nontrivial ideal of $Z$ which is invariant under the restriction of $\beta_x$ to $Z$ for every $x \in \F_2$.

One computes easily that
\[
\begin{split}
\beta_{v_1}(U_1^q)&
=\mu_1^qU_1^q, \\
\beta_{v_1}(U_2^q)&
=U_1^{2q}U_2^q, \\
\beta_{v_2}(U_1^q)&
=U_1^qU_2^{2q}, \\
\beta_{v_2}(U_2^q)&
=\mu_2^qU_2^q.
\end{split}
\]
Set $\nu_1=\mu_1^q$ and $\nu_2=\mu_2^q$, and define homeomorphisms $\varphi_1$ and $\varphi_2$ of $\T^2$ by
\[
\begin{split}
\varphi_1(z_1,z_2)&=(\nu_1z_1,z_1^2z_2), \\
\varphi_2(z_1,z_2)&=(z_1z_2^2,\nu_2z_2).
\end{split}
\]
Identifying $Z$ with $C(\T^2)$ in the obvious way, we get that for $i=1,2$, the restriction of $\beta_{v_i}$ to $Z$ is the map $f \mapsto f\circ \varphi_i$.
By induction we obtain that
\[
\begin{split}
\varphi_1^n(z_1,z_2)&=(\nu_1^nz_1,\nu_1^{n(n-1)}z_1^{2n}z_2), \\
\varphi_2^n(z_1,z_2)&=(\nu_2^{n(n-1)}z_1z_2^{2n},\nu_2^nz_2)
\end{split}
\]
for every $n\in\N$. Then one can use for example \cite[Theorem~6.4]{KuiNie} to deduce that if $\nu_i$ is nontorsion for some $i\in \{1,2\}$ and $(z_1,z_2)\in\T^2$,
then the sequence $(\varphi_i^n(z_1,z_2))_{n=1}^\infty$ is uniformly distributed (sometimes called equidistributed), and therefore dense, in $\T^2$. This implies that if $\nu_i$ is nontorsion, then there is no proper nontrivial ideal of $Z$ which is invariant under $\beta_{v_i}$.
Hence, it follows that $\F_2$ acts on $C^*_r(\Z^2,\sigma_0)$ in a minimal way if $\mu_1$ or $\mu_2$ is nontorsion.
Since $\F_2$ is $C^*$-simple, Proposition~\ref{BryKed}~(i) gives then that $(G,\sigma)$ is $C^*$-simple.
Note that one can easily verify that $(G, \Z^2, \sigma)$ satisfies the relative Kleppner's condition (for any $\sigma$),
so we could instead have invoked Corollary~\ref{sut-torsionfree}. 

Let now $\rho$ be a tracial state on $C_r^*(G,\sigma)$. Then one easily checks that for $m,n, \in \Z$ we have
$\rho(U_1^mU_2^n)=0$ unless both $m$ and $n$ are multiples of $q$. Letting $E_Z$ denote the canonical conditional expectation from $C_r^*(\Z^2,\sigma_0)$ onto $Z$ (see e.g. \cite{Boca}), we get that $\rho = \tilde\rho\circ E_Z$ where $\tilde\rho$ denotes the restriction of $\rho$ to $Z$.
Since $E_Z$ is tracial and equivariant with respect to the action of $\F_2$ on $C_r^*(\Z^2,\sigma_0)$ and its restricted action on $Z$, we obtain that the map $\rho \mapsto \tilde\rho$ gives a one-to-one correspondence between $\F_2$-invariant tracial states on $C^*_r(\Z^2,\sigma_0)$ and $\F_2$-invariant states on $Z$.

Suppose that $\mu_i$ is nontorsion for some $i\in \{1,2\}$. Since the sequence $(\varphi_i^n(z_1,z_2))_{n=1}^\infty$ is uniformly distributed in $\T^2$ for every $(z_1,z_2)\in\T^2$, we get from \cite[Proposition~3.7]{Furst} that $\varphi_i$ is uniquely ergodic on $\T^2$ with respect to the normalized Haar measure $\mu$, i.e., the state $\ell_\mu$ on $Z$ associated to $\mu$ is the only state on $Z$ which is invariant under the restriction of $\beta_{v_i}$ to $Z$.

So if $\mu_1$ or $\mu_2$ is nontorsion, we can conclude that $\ell_\mu$ is the only $\F_2$-invariant state on $Z$. As explained above, this implies that there is only one $\F_2$-invariant tracial state on $C^*_r(\Z^2,\sigma_0)$, namely the canonical tracial state $\tau'$. Applying Proposition~\ref{BryKed}~(ii) (or Corollary~\ref{sut-torsionfree}), we get then that $(G,\sigma)$ has the unique trace property.

Finally, suppose that $\mu_1$ and $\mu_2$ are both torsion. Considering the action of $\F_2$ on $Z$, and the associated action of $\F_2$ on $\T^2$ by homeomorphisms, one easily sees that the orbit $F$ of $(1,1)$ in $\T^2$ under this action is finite. Thus $F$ is a closed $\F_2$-invariant subset of $\T^2 \simeq \op{Prim}(C_r^*(\Z^2, \sigma_0))$. Using Remark~\ref{Prim} we get that $\F_2$ does not act on $C_r^*(\Z^2, \sigma_0)$ in a minimal way. Moreover, we obtain an $\F_2$-invariant state $\ell$ on $Z$ different from $\ell_\mu$ by setting $\ell(f) = \frac{1}{|F|} \sum_{(w_1, w_2) \in F} f(w_1, w_2)$ for $f\in C(\T^2)$. This implies that there are at least two $\F_2$-invariant tracial states on $C_r^*(\Z^2, \sigma_0)$. All in all, we arrive at the conclusion that $(G, \sigma)$ is not $C^*$-simple and does not have the unique trace property in this case.

Summarizing the above discussion, we record the following result:

\begin{proposition} Let $G= \Z^2\rtimes \F_2$ and let $\sigma \in \Z^2(G, \T)$ be determined by $\mu_0$, $\mu_1$, $\mu_2 \in \T$. Then the following conditions are equivalent:
\begin{itemize}
\item[(i)] $(G,\sigma)$ is $C^*$-simple.
\item[(ii)] $(G,\sigma)$ has the unique trace property.
\item[(iii)] At least one of $\mu_0$, $\mu_1$, $\mu_2$ is nontorsion.
\end{itemize}
\end{proposition}

\subsection{Baumslag-Solitar groups}

We recall that the Baumslag-Solitar groups $BS(m,n)$ are groups with presentation $BS(m,n)=\langle a,b \mid ab^m=b^na \rangle$ for nonzero integers $m,n$.
It is well-known that $BS(m,n) \simeq BS(m',n')$ if and only if $(m',n')=(m,n)$, $(-m,-n)$, $(n,m)$, or $(-n,-m)$.
The following holds, cf.~\cite[Equation~5.3]{Rob}:
\begin{itemize}
\item[(a)] $Z(B(m,n))\simeq\Z$ if $m=n$; and $Z(B(m,n))=\{e\}$ if $m \neq n$.

\smallskip \item[(b)] $H^2(B(m,n),\T)\simeq\T$ if $m=n$; and $H^2(B(m,n),\T)=\{1\}$ if $m \neq n$.
\end{itemize}

We therefore fix some $n \geq 2$ and set $G=B(n,n)$.
Note that $Z(G)=\langle b^n \rangle\simeq\Z$, and $G/Z(G) \simeq \Z * \Z_n$ is ICC. Hence, $Z(G)=FC(G)=FCH(G)$.

Let $\varphi \colon G \to \Z^2$ be the homomorphism determined by $\varphi(a)= (1,0)$ and $\varphi(b) = (0,1)$.
Then the kernel of $\varphi$ can be described as
\[
\op{ker}\varphi = \big\langle a^ib^ja^{-i}b^{-j} : i \in \Z\setminus\{0\}, j\in\{1,2,\dotsc,n-1\} \big\rangle \, \simeq \, \F_\infty
\]
For $\omega \in Z^2(\Z^2,\T)$, define the inflation $\op{Inf}\omega \in Z^2(G,\T)$ by $\op{Inf}\omega(x,y)=\omega(\varphi(x),\varphi(y))$.

\begin{lemma}\label{BS-1}
The map $\omega \mapsto \op{Inf}\omega$ induces an isomorphism from $H^2(\Z^2,\T)$ onto $H^2(G,\T)$.
\end{lemma}

\begin{proof}
Set $N=\op{ker}\varphi\simeq\F_\infty$, so that $G/N\simeq\Z^2$, and note that $H^2(N,\T)$ and $H^3(G/N,\T)$ are both trivial.
Therefore we get the following Lyndon-Hochschild-Serre inflation-restriction exact sequence (see e.g.\ \cite[Appendix~2]{PR}):
\begin{gather*}
1 \longrightarrow
\op{Hom}(G/N,\T) \overset{\op{inf}}{\longrightarrow}
\op{Hom}(G,\T) \overset{\op{res}}{\longrightarrow}
\op{Hom}(N,\T)^{G/N} \\ \longrightarrow
H^2(G/N,\T) \overset{\op{Inf}}{\longrightarrow}
H^2(G,\T) \longrightarrow
H^1(G/N,\op{Hom}(N,\T)) \longrightarrow
1
\end{gather*}
It is straightforward to check that $\op{Hom}(N,\T)^{G/N}$ and $H^1(G/N,\op{Hom}(N,\T))$ are trivial,
so we get that $\op{Inf}$ induces an isomorphism.
\end{proof}

For $\lambda\in\T$ we define $\omega_\lambda \in Z^2(\Z^2,\T)$ by $\omega_\lambda(r,s)=\lambda^{r_2s_1}$.

\begin{lemma}\label{BS-2}
Let $\lambda\in\T$ and let $\omega_\lambda\in Z^2(\Z^2, \T)$ be as above. Set $\sigma=\op{Inf}\omega_\lambda$.

Then the following conditions are equivalent:
\begin{itemize}
\item[(i)] $(G,\sigma)$ satisfies Kleppner's condition.
\item[(ii)] $(G,Z(G),\sigma)$ satisfies the relative Kleppner condition.
\item[(iii)] $\lambda$ is nontorsion.
\end{itemize}
\end{lemma}

\begin{proof}
Let $\varphi_1\colon G\to \Z$ be the homomorphism satisfying $\varphi_1(a) = 1$, $\varphi_1(b)=0$, so $\varphi_1(x)$ is the first coordinate of $\varphi(x)$.
Now, since the $G$-conjugacy class of any element in $G\setminus Z(G)$ is infinite, we have that $(G,\sigma)$ satisfies Kleppner's condition if and only if for each $c\in\Z\setminus\{0\}$ there is some $x \in G$ such that $\sigma(b^{cn},x) \neq \sigma(x,b^{cn})$, i.e., such that
\[
1\neq\sigma(b^{cn},x)\,\overline{\sigma(x,b^{cn})}
=\omega_\lambda(\varphi(b^{cn}),\varphi(x))\,\overline{\omega_\lambda(\varphi(x),\varphi(b^{cn}))}
=\lambda^{\varphi_1(x)cn}.
\]
It is then clear that (i) is equivalent to (iii).

Moreover, if $x \in G$, then its $Z(G)$-conjugacy class in $G$ is just $\{x\}$.
Hence, $(G,Z(G),\sigma)$ satisfies the relative Kleppner condition if and only if every $x\in G\setminus Z(G)$ is not $\sigma$-regular w.r.t.\ $Z(G)$. Consider $x\in G\setminus Z(G)$.
Then $\varphi_1(x) \neq 0$ and, as above, we have
\[
\sigma(b^{dn},x)\,\overline{\sigma(x,b^{dn})} = \lambda^{\varphi_1(x)dn}\,
\]
for all $d\in \Z$.
Hence, if $\lambda$ is nontorsion, we see that we can pick $d \in \Z$ such $\sigma(b^{dn},x) \neq \sigma(x,b^{dn})$, so $x$ is not $\sigma$-regular w.r.t.\ $Z(G)$.
This show that (iii) implies (i).
On the other hand, if $\lambda$ has torsion, say $\lambda^m = 1$, then, as $\varphi_1(a^m) = m$, we see that $x= a^m$ is $\sigma$-regular w.r.t.\ $Z(G)$. It follows that (i) implies (iii).
\end{proof}

Using the above lemmas we can prove the following result, which completes \cite[Example~4.6]{BO} where only (i) implies (iii) was explained:
\begin{proposition}\label{BS-prop}
Let $n\geq 2$ and $\sigma \in Z^2(BS(n,n),\T)$.

Then the following are equivalent:
\begin{itemize}
\item[(i)] $(BS(n,n),\sigma)$ satisfies Kleppner's condition.
\item[(ii)] $(BS(n,n),\sigma)$ is $C^*$-simple.
\item[(iii)] $(BS(n,n),\sigma)$ has the unique trace property.
\end{itemize}
Hence, $BS(n,n)$ lies in $\K$.
\end{proposition}

\begin{proof}
Using Lemma~\ref{BS-1} we can assume that $\sigma=\op{Inf}\omega_\lambda$ for some $\lambda \in \T$.
To prove that (i)~$\Rightarrow$~(ii) and (i)~$\Rightarrow$~(iii), we will appeal to Corollary~\ref{sut-ab} with $G=BS(n,n)$, $H=Z(G)\simeq \Z$, and $K=G/H = \Z*\Z_n$.
As a section $s\colon K\to G$ for the quotient map $G\to K=\Z*\Z_n\simeq\langle u,v \mid v^n\rangle$, we choose the obvious map $s$ sending a word in $u$ and $v$ to the corresponding word in $a$ and $b$.
Assume that $(G,\sigma)$ satisfies Kleppner's condition.
Since 
\[
\sigma(b^{cn}, b^{dn}) = \omega_\lambda\big((0,b^{cn}), \, (0, b^{dn})\big) = \lambda^0 = 1
\]
for all $c, d \in \Z$, we have $\sigma' = 1$, so $C_r^*(H, \sigma') = C_r^*(\Z)$ is commutative.
Moreover, Lemma~\ref{BS-2} gives that $(G,H,\sigma)$ satisfies the relative Kleppner condition and it follows from \cite[Proposition~4.3]{BO} that $\tau'$ is the only $K$-invariant tracial state on $C_r^*(H, \sigma')$.
So to apply Corollary~\ref{sut-ab} and obtain that $(G,\sigma)$ is $C^*$-simple with the unique trace property, it only remains to show that $K$ acts on $C_r^*(H,\sigma') \simeq C_r^*(\Z)$ in a minimal way.
One easily computes that the action $\beta$ of $K$ is untwisted and satisfies that
\[
\beta_k(\lambda_{\sigma'}(b^{cn})) = \overline{\lambda}^{\,cn\,\varphi_1(s(k))}\, \lambda_{\sigma'}(b^{cn})
\]
for all $k\in K$ and $c\in \Z$, where $\varphi_1$ is defined as in the proof of Lemma~\ref{BS-2}.
Identifying $C_r^*(H,\sigma') \simeq C_r^*(\Z)$ with $C(\T)$ via Gelfand's transform, we get that each $\beta_k$ is the $*$-automorphism of $C(\T)$ induced by the homeomorphism of $\T$ given by
\[
\phi_k (z) = \lambda^{\,n\varphi_1(s(k))}\,z
\]
for all $z\in \T$.
Since $\varphi_1(s(u^m)) =\varphi_1(a^m) = m$ for every $m \in \Z$, and $\lambda$ is nontorsion (using Lemma~\ref{BS-2}), we see that the orbit $\{\phi_k(z) : k\in K\} $ is dense in $\T$ for every $z \in \T$, so the action of $K$ on $C^*_r(H,\sigma')$ is minimal, as desired.

Since both (ii)~$\Rightarrow$~(i) and (iii)~$\Rightarrow$~(i) always hold, the proof is finished.
\end{proof}

Finally, we can now deduce that $BS(m,n)$ belongs to $\mathcal{K}$ if $\lvert m\rvert,\lvert n\rvert\geq 2$.
Indeed, when $\lvert m\rvert,\lvert n\rvert\geq 2$ and $\lvert m\rvert\neq\lvert n\rvert$,
the group $BS(m,n)$ is $C^*$-simple by \cite[Theorem~4.10]{Iva},
and if $m=-n$ then $BS(m,n)$ is not ICC and has no $2$-cocycles.

\begin{appendix}
\section{On reduced twisted group \texorpdfstring{C$^*$}{C*}-algebras with stable rank one}\label{appendix}

Let $G$ be a discrete group and $\sigma \in Z^2(G, \T)$. We set $\delta = \delta_e$ and let $\lVert \cdot\rVert_2$ denote the usual norm in $\ell^2(G)$. For $a \in \mathcal{B}(\ell^2(G))$ we set
\[
r_2(a) = \limsup_{n\to\infty} \lVert a^n\delta \rVert_2^{1/n}\,.
\]
Since $\lVert b\,\delta \rVert_2 \leq \lVert b\rVert$ for every $b \in \mathcal{B}(\ell^2(G))$, we have $r_2(a) \leq r(a) \leq \lVert a\rVert < \infty $, where $r(a)$ denotes the usual spectral radius of $a \in \mathcal{B}(\ell^2(G))$.

We recall some definitions from \cite{DH}. We let $C_c(G)$ denote the space of all complex-valued functions on $G$ having finite support.
A finite subset $S$ of $G$ is said to have the \emph{$\ell^2$-spectral radius property} if, for every $f\in C_c(G)$ with $\op{supp}(f)\subset S$, we have
\begin{equation}\label{sr2}
r_2\big(\Lambda(f)\big) = r\big(\Lambda(f)\big)\,.
\end{equation}
The group $G$ is said to have the \emph{$\ell^2$-spectral radius property} if every finite subset of $G$ has the $\ell^2$-spectral radius property, that is, if \eqref{sr2} holds for every $f\in C_c(G)$.

\medskip
Dykema and de~la~Harpe show in \cite[Theorem~1.4]{DH} that $C_r^*(G)$ has stable rank one whenever $G$ satisfies the following condition:

\medskip
\emph{\(DH\) \quad For every finite subset $F$ of $G$, there exists $g\in G$ such that $gF$ is semifree
\(i.e., the subsemigroup generated by $gF$ in $G$ is free over $gF$\) and $gF$ has the $\ell^2$-spectral radius property.}

\medskip
The group $G$ is said to have \emph{the free semigroup property} if for every finite subset $F$ of $G$, there exists $g\in G$ such that $gF$ is semifree. An immediate corollary is that $C_r^*(G)$ has stable rank one whenever $G$ is a group having both the free semigroup property and the $\ell^2$-spectral radius property. We will show below that a similar result hold in the twisted case.

\medskip
It will be convenient to introduce some more terminology. We first note that if $f\in C_c(G)$, then $\Lambda_\sigma(f) \delta = f$, so we have
\[
\lVert f\rVert_2 \leq \lVert\Lambda_\sigma(f)\rVert\,.
\]
A finite subset $S$ of $G$ will be said to have \emph{the SR-property w.r.t.\ $\sigma$} if for every $f\in C_c(G)$ with $\op{supp}(f)\subset S$, we have
\[
r\big(\Lambda_\sigma(f)\big) \leq \lVert f\rVert_2\,.
\]
In the case where $\sigma=1$, we just say that $S$ has \emph{the SR-property}.

\begin{theorem}\label{st-rank}
Consider the following conditions:
\begin{itemize} 
\item[(i)] $G$ has the $\ell^2$-spectral radius property and the free semigroup property.
\item[(ii)] $G$ satisfies condition \(DH\).
\item[(iii)] For every finite subset $F$ of $G$, there exists $g\in G$ such that $gF$ has the SR-property.
\item[(iv)] For every finite subset $F$ of $G$, there exists $g\in G$ such that $gF$ has the SR-property w.r.t.\ $\sigma$.
\item[(v)] $C_r^*(G, \sigma)$ has stable rank one.
\end{itemize}
Then we have \textup{(i)}~$\Rightarrow$~\textup{(ii)}~$\Rightarrow$~\textup{(iii)}~$\Rightarrow$~\textup{(iv)}~$\Rightarrow$~\textup{(v)}.
\end{theorem}

The following lemma will be useful in the proof of Theorem~\ref{st-rank}.

\begin{lemma}\label{norm-ineq}
Let $f \in C_c(G)$ and set $a= \Lambda_\sigma(f) \in C_r^*(G, \sigma), \, b=\Lambda(\lvert f\rvert) \in C_r^*(G)$.
Then $\lVert a^n\rVert \leq \lVert b^n\rVert$ for every $n\in \N$.
\end{lemma}

\begin{proof}
We first prove by induction on $n$ that for each $n\in \N$, we have
\begin{equation}\label{twist-ineq}
\lVert a^n\, \xi\rVert_2 \, \leq \lVert\, b^n \, \lvert\xi\rvert \,\rVert_2 \quad \text{for every} \, \, \xi \in \ell^2(G)\,.
\end{equation}
Let $\xi \in \ell^2(G)$. Since
\[
\lvert a \, \xi\rvert = \lvert f \ast_\sigma \xi\rvert \leq \lvert f\rvert \ast \lvert\xi\rvert = b \, \lvert\xi\rvert \,,
\]
we have $\lVert a\, \xi\rVert_2 = \lVert\,\rvert a\, \xi\lvert\,\rVert_2 \, \leq \lVert\, b \, \lvert\xi\rvert \,\rVert_2$, i.e., \eqref{twist-ineq} holds when $n=1$.

\smallskip
Now, assume that \eqref{twist-ineq} holds for some $n\in \N$. Then, for $\xi \in \ell^2(G)$, we get
\[
\lVert a^{n+1}\, \xi\rVert_2 = \lVert a^n\, a\, \xi\rVert_2 \,\leq \lVert \,b^n \, \lvert a\, \xi\rvert\,\rVert_2 \,\leq \, \lVert b^{n+1}\, \lvert \xi\rvert\,\rVert_2\,,
\]
where we have used the induction hypothesis at the second step and the fact that $0\leq b^n \, \lvert a\, \xi\rvert\leq b^n\, b\, \lvert\xi\rvert = b^{n+1} \, \lvert\xi\rvert$ at the third step. This shows that \eqref{twist-ineq} holds for $n+1$, as desired.

From \eqref{twist-ineq}, we get
\[
\lVert a^n\, \xi\rVert_2 \, \leq \lVert\, b^n \, \lvert\xi\rvert \,\rVert_2 \,\leq \, \lVert b^n \rVert \, \lVert\,\lvert\xi\rvert \,\rVert_2 = \lVert b^n \rVert \, \lVert\xi \,\rVert_2
\]
for every $\, \xi \in \ell^2(G)$, and the assertion clearly follows.
\end{proof}
\begin{proof}[Proof of Theorem~\ref{st-rank}]
As already pointed out, (i)~$\Rightarrow$~(ii) is immediate from the definitions.
Next, let $S$ be a finite subset $S$ of $G$. Recall that if $S$ is semifree, then we have $r_2(\Lambda(f)) = \lVert f\rVert_{2}$ for any $f \in C_c(G)$ (cf.\ step two in the proof of Theorem~1.4 in \cite{DH}).
Hence, if $S$ is semifree and has the $\ell^2$-spectral radius property, then we have $r(\Lambda(f)) = r_2(\Lambda(f)) = \lVert f\rVert_{2}$ for every $f\in C_c(G)$ with $\op{supp}(f)\subset S$, so $S$ has the SR-property.
This shows that (ii)~$\Rightarrow$~(iii).

Now, let $S$ be a finite subset of $G$ such that $S$ has the SR-property.
To show that (iii)~$\Rightarrow$~(iv) holds, it suffices to show that $S$ has the SR-property w.r.t.\ $\sigma$.
So consider $f\in C_c(G)$ with $\op{supp}(f)\subset S$ and set $a = \Lambda_\sigma(f)$.
We have to show that $r(a) \leq \lVert f\rVert_2$.
Set $b= \Lambda(\lvert f\rvert) \in C_r^*(G)$.
Since $\op{supp}(\lvert f\rvert)=\op{supp}(f)\subset S$ and $S$ has the SR-property, we get that $r(b) \leq \lVert\,\lvert f\rvert\,\rVert_{2}= \lVert f\rVert_2$.
Thus, we see that it is enough to show that $r(a) \leq r(b)$.
Using the spectral radius formula, this immediately follows from Lemma~\ref{norm-ineq}.

The proof of (iv)~$\Rightarrow$~(v) is an adaptation of the proof of \cite[Theorem~1.4]{DH} (which itself builds upon ideas from \cite{DHR}). For the sake of completeness, we sketch the argument. Assume that (iv) holds and suppose (for contradiction) that $A:=C_r^*(G,\sigma)$ does not have stable rank one. Proceeding as in step three of the proof of \cite[Theorem~1.4]{DH}, we get that there exists some $f \in C_c(G)$ such that
\[
\lVert f\rVert_2\, < \, d\big(\Lambda_\sigma(f), \op{GL}(A)\big)\, ,
\]
where $d\big(x, \op{GL}(A)\big)$ denotes the distance (w.r.t.\ operator norm) from some $x \in A$ to the set of invertible elements in $A$.

Set $a=\Lambda_\sigma(f)$ and $F=\op{supp}(f)$. By assumption, there exists $g \in G$ such that $gF$ has the SR-property w.r.t.\ $\sigma$.
Set $c = \lambda_\sigma(g)\, a \in A$. Clearly, $d\big(c, \, \op{GL}(A)\big) = d\big(a, \, \op{GL}(A)\big)$.
Moreover, since $c= \Lambda_\sigma(f_{g})$, where
\[
f_{g}:= \sum_{h \in F} \,f(h)\sigma(g,h) \,\delta_{gh}\,,
\]
we get that $\lVert f_{g}\rVert_2 = \lVert f\rVert_{2}$ and $\op{supp}(f_{g}) = gF$. Hence, since $gF$ has the SR-property w.r.t.\ $\sigma$, we get that
\[
r(c) \leq \lVert f_{g}\rVert_{2} = \lVert f\rVert_2\,.
\]
We also have $d\big(c, \, \op{GL}(A)\big) \leq \, r(c)$ (as this inequality holds in every unital C$^*$-algebra, cf.\ step one in the proof of \cite[Theorem~1.4]{DH}). Thus, altogether, we get
\[
\lVert f\rVert_{2}\, < \, d\big(a, \, \op{GL}(A)\big) = d\big(c, \, \op{GL}(A)\big) \leq \, r(c) \,\leq \lVert f\rVert_{2} \,,
\]
which gives a contradiction. So $A$ must have stable rank one, that is, (v) holds.
\end{proof}

\begin{remark}
Several examples of groups having both the free semigroup property and the $\ell^2$-spectral radius property are exhibited in \cite{DH}.
If $G$ denotes any of these groups, then Theorem~\ref{st-rank} gives that $C_r^*(G, \sigma)$ has stable rank one for any $\sigma \in Z^2(G,\T)$.
In all these examples, it is known that $G$ is $C^*$-simple (being a Powers group), hence that $(G,\sigma)$ is also $C^*$-simple.
This provides some evidence that it might be true that $C_r^*(G,\sigma)$ has stable rank one whenever $(G,\sigma)$ is $C^*$-simple (cf.~Question~\ref{simp-sr1}).
\end{remark}

\section{On groups with property~\(BP\)}\label{BP}
 We recall from \cite{TD} that 
a group $G$ is said to have \emph{property~\(BP\)} if for every $g \in G \setminus \{e\}$ and $n \in \N, \,n\geq 2$,
there exist $\, g_1, \ldots, g_n \in G$, a subgroup $H$ of $G$, and pairwise disjoint nonempty subsets $T_1,\ldots , T_n \subset H$ such that
\[
g_j \, g \, g_j^{-1}\,\big(H\setminus T_j\big) \,\subset \,T_j
\]
for all $j =1, \ldots, n$.

In \cite[Remark~5.9]{TD}, Tucker-Drob sketches briefly how some arguments of Bekka, Cowling and de la Harpe in \cite{BCH} can be adapted to prove that $G$ has the unique trace property whenever $G$ has property~\(BP\). 
With the kind permission of Tucker-Drob, we give below an expanded version of his proof in the twisted case.

\begin{proposition}
Assume that $G$ has property~\(BP\) and let $\sigma \in Z^2(G,\T)$.
Then $(G,\sigma)$ has the unique trace property.
Moreover, $G$ is ICC and belongs to $\K_{UT}$.
\end{proposition}

\begin{proof}
Let $\psi$ be a tracial state on $A:=C_r^*(G,\sigma)$. To show the first assertion, by continuity of $\psi$ and density of the $*$-subalgebra of $A$ generated by $\lambda_\sigma(G)$, it suffices to show that $\psi\big(\lambda_\sigma(g)\big) = 0$ for all $g\in G\setminus \{e\}$.
\medskip 
Fix $g\in G\setminus\{e\}$ and let $n\in \N,\, n\geq 2$. Pick $g_1, \ldots, g_n$, $H$ and $T_1,\ldots , T_n$ as in the definition of property~\(BP\), and set
\[
a_n = \frac{1}{n} \, \sum_{j=1}^n \, \lambda_\sigma(g_j) \, \lambda_\sigma(g)\, \lambda_\sigma(g_j)^* =
\frac{1}{n}\, \sum_{j=1}^n \, \widetilde{\sigma}(g_j, g) \, \lambda_\sigma(g_j \,g\, g_j^{-1})\,,
\]
where $\widetilde{\sigma}$ is defined as in \eqref{t-sigma}. We will show that
\begin{equation} \label{ineqBP}
\big\lVert\,a_n\,\big\rVert \, \leq \, \frac{2}{\sqrt{n}}\,.
\end{equation}
Using this inequality and the traciality of $\psi$, we then obtain that
\[
\big\lvert\psi\big(\lambda_\sigma(g)\big)\big\rvert = \big\lvert \psi(a_n) \big\rvert \leq \big\lVert a_n\big \rVert \, \leq \, \frac{2}{\sqrt{n}}\,.
\]
Letting $n\to \infty$, we get $\psi\big(\lambda_\sigma(g)\big) = 0 $, as desired.

To show that \eqref{ineqBP} holds, we set $r_j = g_j\, g \,g_j^{-1}$ for each $j= 1, \ldots , n$. The assumption says that $r_j (H\setminus T_j) \subset T_j$ for each $j=1,\ldots, n$. Since $H\setminus T_j \neq \emptyset$ for each $j$ (otherwise the $T_j$'s could not be pairwise disjoint), we see that each $r_j$ belongs to $H$.

Let $\sigma'$ denote the restriction of $\sigma$ to $H\times H$. If $D \subset H$, we let $P_D$ denote the orthogonal projection from $\ell^2(H)$ onto $\ell^2(D)$ (identified as a closed subspace of $\ell^2(H)$). We then have $\lambda_{\sigma'}(h) \, P_D = P_{\,hD}\, \lambda_{\sigma'}(h)$ for all $h \in H$. Note also that, since $r_j (H\setminus T_j) \cap (H\setminus T_j) = \emptyset$, we have $P_{r_j(H\setminus T_j)} \, P_{H\setminus T_j} = 0$ for each $j=1,\ldots, n$.

\smallskip
Set $a'_n= \frac{1}{n}\, \sum_{j=1}^n \, \widetilde{\sigma}(g_j, g) \, \lambda_{\sigma'}(r_j) \in C_r^*(H, \sigma')\,.$
To estimate $\lVert a'_n \rVert$, let $\xi, \eta \in \ell^2(H)$.
Using the triangle inequality, the remarks above and the Cauchy-Schwarz inequality, we get
\begin{align*}
\big\lvert \langle \widetilde{\sigma}(g_j, g)\,\lambda_{\sigma'}(r_j) \, \xi, \, \eta\rangle\big\rvert
& = \, \big\lvert \langle \lambda_{\sigma'}(r_j) \, \xi, \, \eta\rangle\big\rvert \\
& \leq \, \big\lvert \langle \,\lambda_{\sigma'}(r_j) \,P_{T_j}\, \xi, \, \eta\rangle\big\rvert
+ \big\lvert \langle \lambda_{\sigma'}(r_j) \, P_{H\setminus T_j}\,\xi, \, \eta\rangle\big\rvert \\
& \leq \, \big\lvert \langle \lambda_{\sigma'}(r_j) \,P_{T_j}\, \xi, \, \eta\rangle\big\rvert
+ \big\lvert \langle P_{r_j(H\setminus T_j)}\, \lambda_{\sigma'}(r_j) \,\xi, \, \eta\rangle\big\rvert \\
& = \, \big\lvert \langle \lambda_{\sigma'}(r_j) \,P_{T_j}\, \xi, \, \eta\rangle\big\rvert
+ \big\lvert \langle P_{r_j(H\setminus T_j)}\, \lambda_{\sigma'}(r_j) \,\xi, \,P_{T_j} \eta\rangle\big\rvert \\
& \leq \,\lVert P_{T_j} \, \xi \rVert\, \lVert\eta\rVert + \, \lVert\xi\rVert\, \lVert P_{T_j} \, \eta\lVert
\end{align*}
for each $j=1, \ldots, n$. Since the $T_j$'s are pairwise disjoint, this gives
\begin{gather*}
\sum_{j=1}^n \, \big\lvert \langle \widetilde{\sigma}(g_j, g)\, \lambda_{\sigma'}(r_j) \, \xi, \, \eta\rangle\big\rvert \, \leq \, \Big( \lVert\eta\rVert \, \sum_{j=1}^{n}\,\lVert P_{T_j} \, \xi \rVert + \lVert\xi\rVert\, \sum_{j=1}^n\lVert P_{T_j} \, \eta\rVert\Big)
\\
\leq \, \Big( \sqrt{n}\, \lVert\eta\rVert \, \Big(\sum_{j=1}^{n}\,\lVert P_{T_j} \, \xi \rVert^2\Big)^{1/2} + \sqrt{n}\, \lVert\xi\rVert \, \Big(\sum_{j=1}^{n}\,\lVert P_{T_j} \, \eta \rVert^2\Big)^{1/2}\,\Big) \, \leq \, 2\sqrt{n}\, \lVert\xi\rVert\, \lVert\eta\rVert\,.
\end{gather*}
Thus we get 
\[
\big\lvert \big\langle a'_n\, \xi, \, \eta\big\rangle\big\rvert \,\leq\, \frac{1}{n}\, \sum_{j=1}^n \, \Big\lvert \Big\langle \widetilde{\sigma}(g_j, g)\,\lambda_{\sigma'}(r_j) \, \xi, \, \eta\Big\rangle\Big\rvert \, \leq \, \frac{1}{n} \, 2\sqrt{n}\, \lVert\xi\rVert\, \lVert\eta\rVert\, =\, \frac{2}{\sqrt{n}}\, \lVert\xi\rVert\, \lVert\eta\rVert\,.
\]
It follows that $\, \lVert a_n\rVert=\lVert a'_n\rVert \, \leq \, 2/\sqrt{n}$, that is, \eqref{ineqBP} holds, and the proof of the first assertion is finished. Since this assertion is true for any $\sigma \in Z^2(G,\T)$, the second assertion follows readily.
\end{proof}

\section{On decay properties and uniqueness of the trace} \label{decay}
A recent result of Gong says that if a group $G$ has Jolissaint's property RD \cite{Jolissaint} (with respect to some length function $L$), and every nontrivial conjugacy class of $G$ has superpolynomial growth (w.r.t.~$L$), then $G$ has the unique trace property (see \cite[Theorem~3.11]{Gong1}). We give below a generalized version of her result.

Consider $\kappa\colon G\to [1,\infty)$. For $\xi\colon G\to \C$, set $\lVert \xi\rVert_{2,\kappa}=\Big(\sum_{g\in G} \lvert\xi(g) \kappa(g)\rvert^2\Big)^{1/2} \in [0,\infty]$.\\
Let $\sigma \in Z^2(G,\T)$. We will say that $(G, \sigma)$ is \emph{$\kappa$-decaying} if there exists some $M>0$ such that
\[
\lVert\Lambda_\sigma(f)\rVert \leq M\, \lVert f\rVert_{2,\kappa}
\]
for every function $f\colon G\to \C$ having finite support. It is easy to see that this definition agrees with the one given in \cite{BC}. 
When $(G,1)$ is $\kappa$-decaying, we will just say that $G$ is $\kappa$-decaying. We note that if $L$ is a length function on $G$, then $G$ has \emph{property RD} (w.r.t.~$L$) in the sense of \cite{Jolissaint} if and only if there exists some $s>0$ such that $G$ is $(1+L)^s$-decaying.

According to \cite[Theorem~3.5 and Proposition~3.7]{BC}, we have:
\begin{itemize}
\item[(i)] if $G$ is $\kappa$-decaying, then $(G,\sigma)$ is $\kappa$-decaying;
\item[(ii)] if $(G,\sigma)$ is $\kappa$-decaying, then the series $\sum_{g\in G} \xi(g)\lambda_\sigma(g)$ is convergent w.r.t.\ the operator-norm in $C_r^*(G,\sigma)$ whenever $\lVert\xi\rVert_{2,\kappa} < \infty$.
\end{itemize}

Assume now that $\kappa\colon G\to [1,\infty)$ is proper (so $G$ is countable). Let $C$ be a subset of $G$. For each $k\in \N$, set $C_k= C \cap \{g\in G : k-1< \kappa(g) \leq k\}$. 
 We will say that $C$ has \emph{superpolynomial growth} (w.r.t.~$\kappa$) if for every (real) polynomial $P$ there exists an infinite subset $K$ of $\N$ such that $\lvert C_k\rvert > P(k)$ for all $k\in K$.

\begin{theorem}
Assume that $\kappa\colon G\to [1,\infty)$ is a proper map, $(G,\sigma)$ is $\kappa$-decaying and every nontrivial $\sigma$-regular conjugacy class in $G$ has superpolynomial growth (w.r.t.~$\kappa$). Then $(G, \sigma)$ has the unique trace property.
\end{theorem}
\begin{proof}
A major part of the proof is an adaptation of the proof of \cite[Lemma~3.9]{Gong1}. Let $\omega$ be a tracial state on $C_r^*(G,\sigma)$. It suffices to show that $\omega(\lambda_\sigma(g)) = 0$ for every $g\in G\setminus\{e\}$. Assume first that $g\in G$ is not $\sigma$-regular. Let then $h\in G$ be such that $h$ commutes with $g$ and $\sigma(h,g)\neq \sigma(g,h)$. We then have $\widetilde\sigma(h,g) \neq 1$ and
\[
\omega(\lambda_\sigma(g)) = \omega\big(\lambda_\sigma(h)\lambda_\sigma(g)\lambda_\sigma(h)^*\big) = \widetilde\sigma(h,g)\,\omega(\lambda_\sigma(g)),
\]
so it follows that $\omega(\lambda_\sigma(g)) =0$.

Next, assume that $g\in G\setminus \{e\} $ is $\sigma$-regular. Set $C=C_G(g)$ (the conjugacy class of $g$ in $G$) and let $C_k$ be defined as above for each $k\in \N$. Since $C$ has superpolynomial growth (w.r.t.~$\kappa$), we can find an increasing sequence $1 < k_1 < k_2 < \cdots$ in $\N$ such that $c_j:=\lvert C_{k_j}\rvert > (k_j)^{4j}$ for every $j\in \N$. 

Using equation \eqref{sigma-conjugate} we get that for each $h\in C$, there exists some $\gamma_h \in \T$ such that $\omega(\lambda_\sigma(h)) = \gamma_h \, \omega(\lambda_\sigma(g))$. Define then $\xi\colon G\to \C$ by
\[
\xi(h) = \begin{cases} \overline{\gamma_h} \, c_{j}^{-5/8} & \text{if }  h \in C_{k_j} \text{ for some } j \in \N,\\
\hspace{3ex} 0 & \text{otherwise.}\end{cases}
\] 
Then 
\begin{align*}
\lVert \xi\rVert_{2,\kappa}^2 &= \sum_{j\in \N} \,\sum_{h \in C_{k_j}} \,\big\lvert\overline{\gamma_h}\, c_{j}^{-5/8}\, \kappa(h)\big\rvert^2
\,\leq\, \sum_{j\in \N} \,\sum_{h \in C_{k_j}} \,c_{j}^{-10/8}\, k_j^{\,2}\\
&= \sum_{j\in \N} \, \,c_{j}^{-1/4}\, k_j^{\,2}
\,\leq \, \sum_{j\in \N} \, \,k_j^{-j}\, k_j^{\,2} \,< \,\infty\,.
\end{align*}
Since $(G,\sigma)$ is $\kappa$-decaying, we get that $\sum_{h\in G} \xi(h)\,\lambda_\sigma(h)$ converges in operator-norm to some $x \in C_r^*(G,\sigma)$. Thus, by continuity of $\omega$, we get
\[
\sum_{h \in G} \xi(h)\, \omega\big(\lambda_\sigma(h)\big)= \omega(x).
\]
But 
\begin{align*}\xi(h) \,\omega(\lambda_\sigma(h)\big)&= 
\begin{cases} \overline{\gamma_h} \, c_{j}^{-5/8}\,\gamma_h \, \omega(\lambda_\sigma(g)) & \text{if }  h \in C_{k_j} \text{ for some } j \in \N,\\
\hspace{10ex} 0 & \text{otherwise.}\end{cases} \\ &= \begin{cases} c_{j}^{-5/8}\,\omega(\lambda_\sigma(g)) & \text{if }  h \in C_{k_j} \text{ for some } j \in \N,\\
\hspace{10ex} 0 & \text{otherwise.}\end{cases} 
\end{align*} 
Hence
\[
\sum_{h \in G} \xi(h) \omega(\lambda_\sigma(h)) =\omega(\lambda_\sigma(g))\,\sum_{j\in \N}\sum_{h\in C_{k_j}} c_j^{-5/8}= \omega(\lambda_\sigma(g))\,\sum_{j\in \N} c_j^{3/8}\,.
\]
As $\sum_{j\in \N} c_j^{3/8} =\infty$, we see that we must have $\omega(\lambda_\sigma(g)) =0$. 

Thus, altogether, we have shown that $\omega(\lambda_\sigma(g)) =0$ for all $g\in G\setminus\{e\}$, which implies that $\omega$ agrees with the canonical tracial state on $C_r^*(G,\sigma)$.
\end{proof}

\end{appendix}

\section*{Acknowledgements}

We thank the referee for many helpful suggestions.
The second author is funded by the Research Council of Norway through FRINATEK, project no.~240913. Part of the work was done when he was affiliated with Arizona State University in Tempe, and in particular during the visit of the first author to ASU in April 2015. We are both very grateful to the operator algebra group at ASU for their kind hospitality.

\bibliographystyle{plain}

\end{document}